\documentclass[english]{article}
\usepackage{color}
\usepackage{babel}
\usepackage{array}
\usepackage{longtable}
\usepackage{mathtools}
\usepackage{url}
\usepackage{amsmath}
\usepackage{amsthm}
\usepackage{amssymb}
\usepackage{cancel}
\usepackage{mathdots}
\usepackage{stmaryrd}
\usepackage{stackrel}
\usepackage{undertilde}
\usepackage{graphicx}
\PassOptionsToPackage{version=3}{mhchem}
\usepackage{mhchem}
\usepackage{esint}
\usepackage[numbers]{natbib}
\usepackage[unicode=true,
 bookmarks=false,
 breaklinks=false,pdfborder={0 0 1},backref=false,colorlinks=false]
 {hyperref}
\usepackage{breakurl}

\makeatletter

\providecommand{\tabularnewline}{\\}

\theoremstyle{plain}
\newtheorem{thm}{\protect\theoremname}
  \theoremstyle{plain}
  \newtheorem{prop}{\protect\propositionname}
  \theoremstyle{definition}
  \newtheorem{defn}{\protect\definitionname}
  \theoremstyle{remark}
  \newtheorem*{rem*}{\protect\remarkname}

\setkeys{Gin}{width=\linewidth}
\usepackage{etex}
\usepackage{color}
\usepackage{fontenc}
\usepackage{graphicx}
\usepackage{hyperref}
\usepackage{url}
\usepackage{breakurl}
\usepackage{epstopdf}
\usepackage{fullpage,graphicx,psfrag,verbatim}
\usepackage[usenames,dvipsnames]{pstricks}
\usepackage{epsfig}
\usepackage{pst-grad} 
\usepackage{pst-plot}
\usepackage{todonotes}
\usepackage[small,bf]{caption}
\usepackage{longtable}

\newtheorem{assumption}{Assumption}
\newcommand{\ones}{\mathbf 1}
\newcommand{\reals}{{\mbox{\bf R}}}
\newcommand{\integers}{{\mbox{\bf Z}}}

\newcommand{\card}{\mathop{\bf card}}

\newcommand{\conv}{\mathop{\bf conv}}

\newcommand{\argmin}{\mathop{\rm argmin}}

\newcommand{\epi}{\mathop{\bf epi}} 
\newcommand{\env}{\mathop{\bf env}} 

\date{}

\makeatother

  \providecommand{\definitionname}{Definition}
  \providecommand{\propositionname}{Proposition}
  \providecommand{\remarkname}{Remark}
\providecommand{\theoremname}{Theorem}

\begin{document}

\title{A Two-Step Linear Programming Model for Energy-Efficient Timetables
in Metro Railway Networks}

\author{Shuvomoy Das Gupta\thanks{\protect\url{shuvomoy.dasgupta@mail.utoronto.ca}, Department of Electrical
\& Computer Engineering, University of Toronto, 10 King's College
Road, Toronto, Ontario, Canada}\and J. Kevin Tobin\thanks{\protect\url{Kevin.Tobin@thalesgroup.com},Thales Canada Inc., 105
Moatfield Drive, Toronto, Ontario, Canada}\and Lacra Pavel\thanks{\protect\url{pavel@control.toronto.edu}, Department of Electrical
\& Computer Engineering, University of Toronto, 10 King's College
Road, Toronto, Ontario, Canada}\and}
\maketitle
\begin{abstract}
In this paper we propose a novel two-step linear optimization model
to calculate energy-efficient timetables in metro railway networks.
The resultant timetable minimizes the total energy consumed by all
trains and maximizes the utilization of regenerative energy produced by braking trains, subject to the constraints in the railway network.
In contrast to other existing models, which are $\mathcal{NP}$-hard,
our model is computationally the most tractable one being a linear
program.  We apply our optimization model to different instances of
service PES2-SFM2 of line 8 of Shanghai Metro network spanning a full
service period of one day (18 hours) with thousands of active trains.
For every instance, our model finds an optimal timetable very quickly
(largest runtime being less than 13s) with significant reduction in
effective energy consumption (the worst case being 19.27\%). Code
based on the model has been integrated with \texttt{Thales Timetable
Compiler} - the industrial timetable compiler of Thales Inc that has
the largest installed base of communication-based train control systems
worldwide.
\end{abstract}


\section{Introduction}

\subsection{Background and motivation}

Efficient energy management of electric vehicles using mathematical
optimization has gained a lot of attention in recent years \citep{Patil2012,Nuesch2012,Mura2013,Bashash2011,saber2010intelligent}.
When a train makes a trip from an origin platform to a destination
platform, its optimal speed profile consists of four phases: 1) maximum
acceleration, 2) speed hold, 3) coast and 4) maximum brake \citealp{Howlett1995},
as shown in Figure~\ref{Fig:SpeedProfileOfTrain} in a qualitative
manner. 
\begin{figure*}[htp]
\includegraphics[scale=0.5]{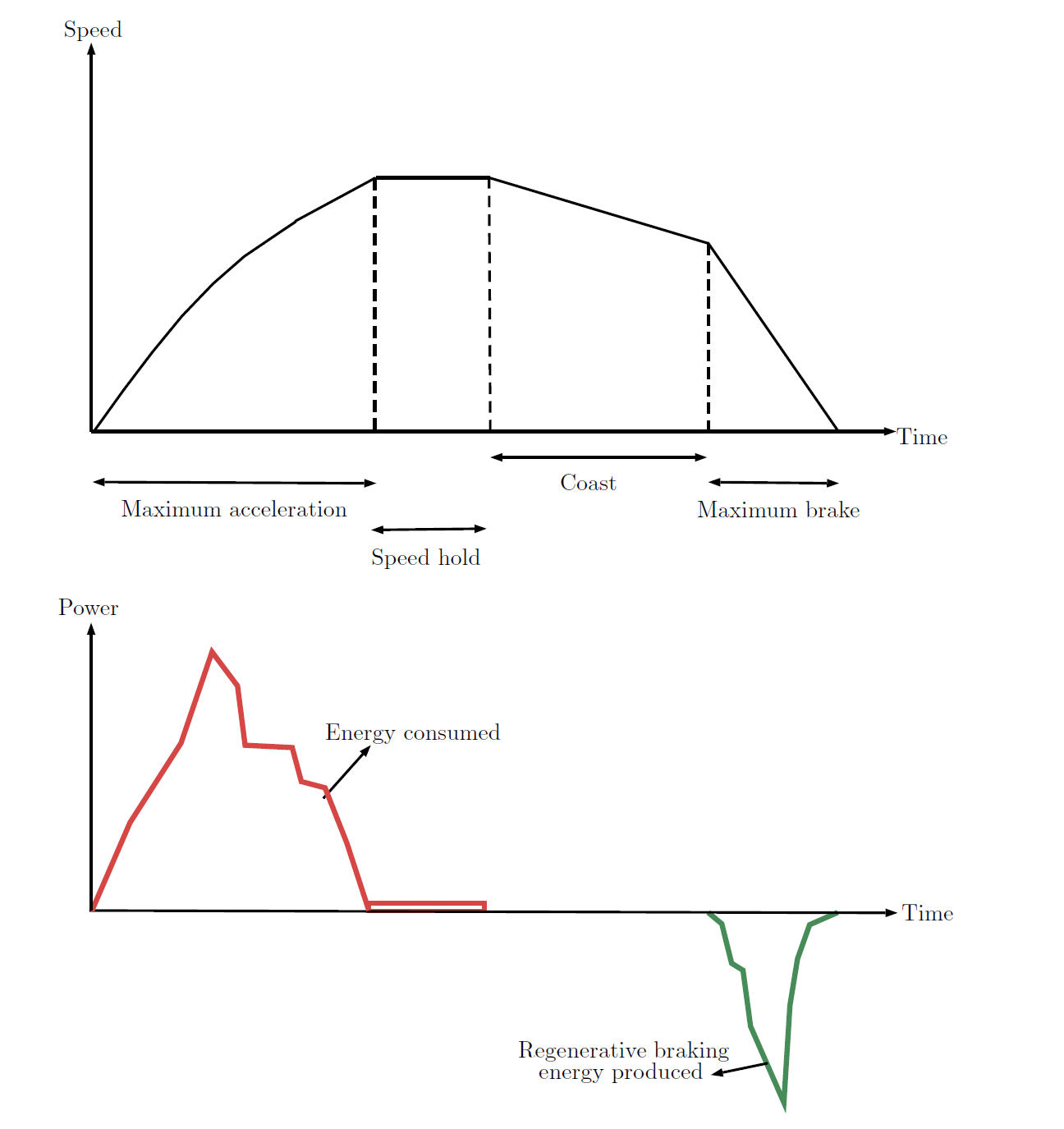} \caption{Optimal speed profile of a train}
\label{Fig:SpeedProfileOfTrain} 
\end{figure*}
Most of the energy required by the train is consumed during the accelerating
phase. During the speed holding phase the energy consumption is negligible
compared to accelerating phase, and during the coasting phase there
is no need for energy. When the train brakes, it produces energy by
using a regenerative braking mechanism. This energy is called regenerative
braking energy. Calculating energy-efficient timetables for trains
in railway networks is a relevant problem in this regard. Electricity
is the main source of energy for trains in most modern railway networks;
in such networks, a train is equipped with a regenerative braking
mechanism that allows it to produce electrical energy during its braking
phase. In this paper, we formulate a two-step linear optimization
model to obtain an energy-efficient timetable for a metro railway
network. The timetable schedules the arrival time and the departure
time of each train to and from the platforms it visits such that the
total electrical energy consumed is minimized and the utilization
of produced regenerative energy is maximized.

\subsection{Related work}

The general timetabling problem in a metro railway network has been
studied extensively over the past three decades \citep{Harrod2012}.
However, very few results exist that can calculate energy-efficient
timetables. Now we discuss the related research. We classify the related
work as follows. The first two papers are mixed integer programming
model, the next three are models based on meta-heuristics and the
last one is an analytical study.

A Mixed Integer Programming (MIP) model, applicable only to single
train-lines, is proposed by Pe�a-Alcaraz et al.\ \citep{Pena-Alcaraz2012}
to maximize the total duration of all possible synchronization processes
between all possible train pairs. The model is then applied successfully
to line three of the Madrid underground system. However, the model
can have some drawbacks. First, considering all train pairs in the
objective will result in a computationally intractable problem even
for a moderate sized railway network. Second, for a train pair in
which the associated trains are far apart from each other, most, if
not all, of the regenerative energy will be lost due to the transmission
loss of the overhead contact line. Finally, the model assumes that
the durations of braking and accelerating phases stay the same with
varying trip times, which is not the case in reality.

The work in \citep{DasGuptaACC15} proposes a more tractable MIP model,
applicable to any railway network, by considering only train pairs
suitable for regenerative energy transfer. The optimization model
is applied numerically to the Dockland Light Railway and shows a significant
increase in the total duration of the synchronization process. Although
such increase, in principle, may increase the total savings in regenerative
energy, the actual energy saving is not directly addressed. Similar
to \citep{Pena-Alcaraz2012}, this model too, assumes that even if
the trip time changes, the duration of the associated braking and
accelerating stay the same.

Other relevant works implement meta-heuristics. The work in \citep{Li2014}
implements genetic algorithm to calculate timetables that maximize
the utilization of regenerative energy while minimizing the tractive
energy of the trains. Numerical studies for the model is implemented
to Beijing Metro Yizhuang Line of China showing notable increase in
energy efficiency. The work in \citep{Yang2013} presents a cooperative
integer programming model to utilize the use of regenerative energy
of trains and proposes genetic algorithm to solve it. Similar to \citep{Li2014},
this numerical studies have been performed to Beijing Metro Yizhuang
Line of China, though the improvement is stated in the increase in
overlapping time only. The work in \citep{Le2014} presents a nonlinear
integer programming model which is solved using simulated annealing.
The numerical experiments have been conducted for the island line
of the mass transit system in Hong Kong.

An insightful analytical study of a periodic railway schedule appears
in \citep{aLi2014a}. The model uses the KKT conditions to calculate
and analyze the properties of an energy-efficient timetable. The resultant
analytical model is then applied to Beijing Metro Yizhuang Line of
China numerically, which shows that the model can reduce the net energy
consumption considerably.

\begin{figure}[htp]
\includegraphics{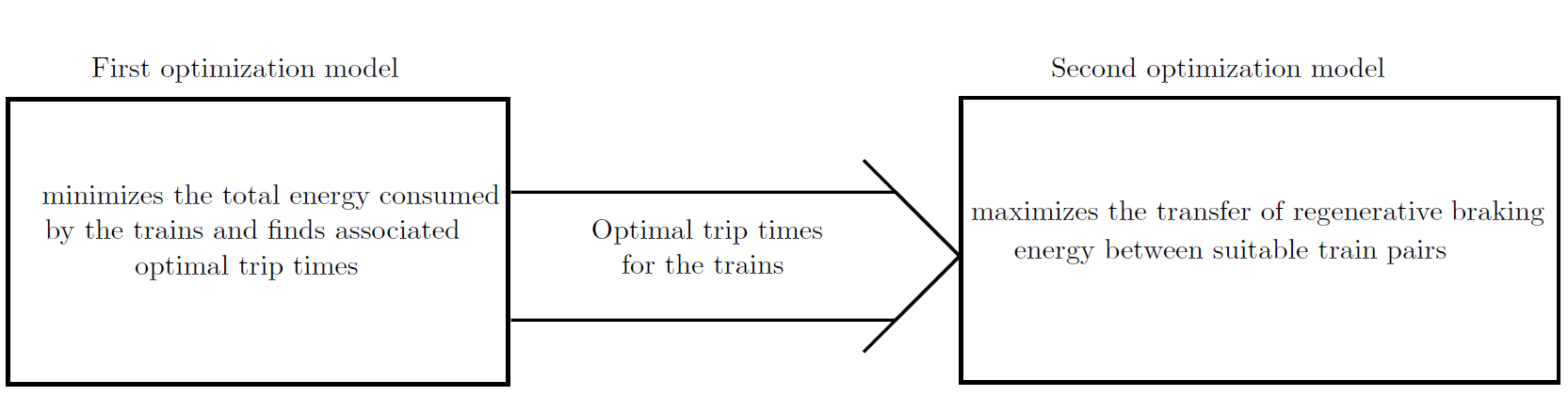}

\caption{Flow-chart of the two steps of the optimization model}
\label{fig:flow-chart} 
\end{figure}

\subsection{Contributions}

We propose a novel two-step linear optimization model to calculate
an energy-efficient railway timetable. The \emph{first} optimization
model minimizes the total energy consumed by all trains subject to
the constraints present in the railway network. The problem can be
formulated as a linear program, with the optimal value attained by
an integral vector. The \emph{second} optimization model uses the
optimal trip time from the first optimization model and maximizes
the transfer of regenerative braking energy between suitable train
pairs. Both the steps of our optimization model are linear programs,
whereas the optimization models in related works are $\mathcal{NP}$-hard.
A flow-chart of the two steps of the optimization model is shown in
Figure~\ref{fig:flow-chart}. Our model can calculate energy-efficient
railway timetables for large scale networks in a short CPU time. Code
based on the model has been embedded with the railway timetable compiler
(\texttt{Thales Timetable Compiler}) of Thales Inc, which has the
largest installed base of communication-based train control systems
worldwide. \texttt{Thales Timetable Compiler} is used by many railway
management systems worldwide including: Docklands Light Railway in
London, UK., the West Rail Line and Ma On Shan Line in Hong Kong,
the Red Line and Green Line in Dubai, the Kelana Jaya Line in Kuala
Lumpur.

This paper is organized as follows. In Section~\ref{notationAndNotions}
we describe the notation used, and then in Section~\ref{modellingConstraints}
we model and justify the constraints present in the railway network.
The first optimization model is presented in Section~\ref{Preli_Opt}.
Section~\ref{Final_Opt} formulates the second optimization problem
that additionally maximizes the utilization of regenerative braking
energy. Section \ref{sec:lim} describes the limitations of the optimization
model. In Section~\ref{numStud} we apply our model to different
instances of an existing metro railway network spanning a full working
day and describe the results. Section~\ref{Conclusion} presents
the conclusion.

\section{Notation and notions}

\label{notationAndNotions} Every set described in this paper is strictly
ordered and finite unless otherwise specified. The set-cardinality
(number of elements of the set) and the $i$th element of such a set
$C$ is denoted by $|C|$ and $C(i)$ respectively. The set of real
numbers and integers are expressed by $\reals$ and $\integers$ respectively;
subscripts $+$ and $++$ attached with either set denote non-negativity
and positivity of the elements respectively. A column vector with
all components one is denoted by $\ones$. The symbol $\preceq$ stands
for componentwise inequality between two vectors and the symbol $\wedge$
stands for conjunction. The number of nonzero components of a vector
$x$ is called cardinality of that vector and is denoted by $\card(x)$.
Note that, cardinality of a vector is different from set-cardinality.
The $i$th unit vector $e_{i}$ is the vector with all components
zero except for the $i$th component which is one. The epigraph of
a function $f:C\rightarrow\reals$ (where C is any set) denoted by
$\epi f$ is the set of input-output pairs that $f$ can achieve along
with anything above, \emph{i.e.,} 
\[
\epi f=\{(x,t)\in C\times\reals\mid x\in C,t\geq f(x)\}.
\]
The convex hull of any set $C$, denoted by $\conv C$, is the set
containing all convex combinations of points in $C$. Consequently,
if $C$ is nonconvex, then its best convex outer approximation is
$\conv C$, as it is the smallest set containing $C$.

The set of all platforms in a railway network is indicated by $\mathcal{N}$.
A directed arc between two distinct and non-opposite platforms is
called a track. The set of all tracks is represented by $\mathcal{A}$.
The directed graph of the railway network is expressed by $(\mathcal{N},\mathcal{A})$.
The set of all trains is denoted by $\mathcal{T}$, where $|\mathcal{T}|$
is fixed. The sets of all platforms and all tracks visited by a train
$t$ in chronological order are denoted by $\mathcal{N}^{t}\subseteq\mathcal{N}$
and $\mathcal{A}^{t}\subseteq\mathcal{A}$ respectively. The decision
variables are the train arrival and departure times, to and from the
associated platforms, respectively. Let $a_{i}^{t}$ and $d_{i}^{t}$
be the arrival time and the departure time of the train $t\in\mathcal{T}$
to and from the platform $i\in\mathcal{N}^{t}$.

Table \ref{my-label} contains the list of symbols used in this paper.

\begin{longtable}{ll}
\caption{List of symbols}
\tabularnewline
$|C|$  & The number of elements of a finite countable set $C$ \tabularnewline
$C(i)$  & The $i$th element of a finite countable set $C$ \tabularnewline
$\reals$  & The set of real numbers \tabularnewline
$e_{i}$  & The $i$th unit vector \tabularnewline
$\epi f$  & The epigraph of a function $f$ \tabularnewline
$\card(x)$  & The number of nonzero elements in a vector $x\in\reals^{n}$ \tabularnewline
$\mathcal{N}$  & The set of all platforms in a railway network \tabularnewline
$\mathcal{A}$  & The set of all tracks \tabularnewline
$\mathcal{T}$  & The set of all trains \tabularnewline
$\mathcal{N}^{t}$  & The set of all platforms visited by a train $t$ in chronological
order \tabularnewline
$\mathcal{A}^{t}$  & The set of all tracks visited by a train $t$ in chronological order \tabularnewline
$a_{i}^{t}$  & The arrival time of train $t$ at platform $i$ \tabularnewline
$d_{i}^{t}$  & The departure time of train $t$ from platform $i$ \tabularnewline
$[\underline{\tau}_{ij}^{t},\overline{\tau}_{ij}^{t}]$  & The trip time window for train $t$ from platform $i$ to platform
$j$ \tabularnewline
$\mathcal{B}_{ij}$  & The set of all train pairs involved in turn-around events on crossing-over
$(i,j)$ \tabularnewline
$[\underline{\kappa}_{ij}^{tt'},\overline{\kappa}_{ij}^{tt'}]$  & The trip time window for train $t$ on the crossing-over $(i,j)$ \tabularnewline
$[\underline{\delta}_{i}^{t},\overline{\delta}_{i}^{t}]$  & The dwell time window for train $t$ at platform $i$ \tabularnewline
$\chi$  & The set of all platform pairs situated at the same interchange stations \tabularnewline
$\mathcal{C}_{ij}$  & The set of connecting train pairs for a platform pair $(i,j)\in\chi$ \tabularnewline
$[\underline{\chi}_{ij}^{tt'},\overline{\chi}_{ij}^{tt'}]$  & The connection window between train $t$ at platform $i$ and and
train $t'$ at platform $j$ \tabularnewline
$[0,m]$  & The time bound for any event time in the timetable \tabularnewline
$\mathcal{H}_{ij}$  & The set of train-pairs who move along that track $(i,j)$ \tabularnewline
$[\underline{h}_{i}^{tt'},\overline{h}_{i}^{tt'}]$  & The headway time window between train $t$ and $t'$ at or from platform
$i$ \tabularnewline
$[\underline{\tau}_{\mathcal{P}}^{t},\overline{\tau}_{\mathcal{P}}^{t}]$  & The total travel time window for train $t$ to traverse its train
path \tabularnewline
$\bar{\mathcal{N}}$  & The set of all nodes in the constraint graph \tabularnewline
$\bar{\mathcal{A}}$  & The set of arcs in the constraint graph \tabularnewline
$\bar{\mathcal{A}}_{\textrm{trip}}$  & The set of all arcs associated with trip time constraints \tabularnewline
$x_{i}$  & The arrival or departure time of some train from a platform in the
constraint graph \tabularnewline
$[l_{ij},u_{ij}]$  & The time window associated with arc $(i,j)$ of the constraint graph \tabularnewline
$f_{ij}$  & Energy consumption associated with the trip $(i,j)\in\bar{\mathcal{A}}_{\textrm{trip}}$ \tabularnewline
$c_{ij}(x_{i}-x_{j})+b_{ij}$  & Affine approximation for $f_{ij}$ \tabularnewline
$\left((\bar{a}_{i}^{t},\bar{d}_{i}^{t})_{i\in\mathcal{N}}\right)_{t\in\mathcal{T}}$  & Solution to step one optimization model \tabularnewline
$\Omega$  & The set containing all opposite platform pairs powered by the same
electrical substations \tabularnewline
$\mathcal{T}_{i}$  & The set of all trains which arrive at, dwell and then depart from
platform $i$ \tabularnewline
$\overset{\rightharpoonup}{t}$  & Temporally closest train to the right of train $t$ \tabularnewline
$\overset{\leftharpoonup}{t}$  & Temporally closest train to the left of train $t$ \tabularnewline
$\tilde{t}$  & Temporally closest train to train $t$ \tabularnewline
$\triangledown_{i}^{t}$  & The relative distance of $a_{i}^{t}$ from the regenerative alignment
point \tabularnewline
$\vartriangle_{j}^{t}$  & The relative distance of the consumptive alignment point from $d_{i}^{t}$ \tabularnewline
$\mathcal{E}$  & The set of all synchronization processes between suitable train pairs \tabularnewline
$\overset{\rightharpoonup}{\mathcal{E}}$  & A subset of $\mathcal{E}$ containing elements of the form $(i,j,t,\overset{\rightharpoonup}{t})$ \tabularnewline
$\overset{\leftharpoonup}{\mathcal{E}}$  & A subset of $\mathcal{E}$ containing elements of the form $(i,j,t,\overset{\leftharpoonup}{t})$ \tabularnewline
$\env f$  & Convex envelope of function $f$ \tabularnewline
\label{my-label}  & \tabularnewline
\end{longtable}

\section{Modelling the constraint set}

\label{modellingConstraints} In this section we describe the constraint
set for our optimization model. This comprises the feasibility constraints
for a railway network of arbitrary topology, and the domain of the
decision variables. In every active railway network, the railway management
has an operating feasible timetable; we use the sequence of the trains
from that timetable. The lower and upper bound of the constraints
are integers representing time in seconds.

\subsection{Trip time constraint}

The trip time constraints play the most important role in train energy
consumption and regenerative energy production. These can be of two
types as follows.

\subsubsection{Trip time constraint associated with a track}

Consider the trip of any train $t\in\mathcal{T}$ from platform $i$
to platform $j$ along the track $(i,j)\in\mathcal{A}^{t}$. The train
$t$ departs from platform $i$ at time $d_{i}^{t}$, arrives at platform
$j$ at time $a_{j}^{t}$, and it can have a trip time between $\underline{\tau}_{ij}^{t}$
and $\overline{\tau}_{ij}^{t}$. The trip time constraint can be written
as follows: 
\begin{align}
\forall t\in\mathcal{T},\quad\forall(i,j)\in\mathcal{A}^{t},\quad\underline{\tau}_{ij}^{t}\leq a_{j}^{t}-d_{i}^{t}\leq\overline{\tau}_{ij}^{t}.\label{eq:TripTimeConstraint}
\end{align}

\subsubsection{Trip time constraint associated with a crossing-over}

A crossing-over is a special type of directed arc that connects two\emph{
train-lines}, where a train-line is a directed path with the set of
nodes representing non-opposite platforms and the set of arcs representing
non-opposite tracks. If after arriving at the terminal platform of
a train-line, a train turns around by traversing the crossing-over
and starts travelling through another train-line, then the same physical
train is treated and labelled functionally as two different trains
by the railway management \citep[page~41]{Peeters2003}. Let $\varphi$
be the set of all crossing-overs, where turn-around events occur.
Consider any crossing-over $(i,j)\in\varphi$, where the platforms
$i$ and $j$ are situated on different train-lines. Let $\mathcal{B}_{ij}$
be the set of all train pairs involved in corresponding turn-around
events on the crossing-over $(i,j)$. Let $(t,t')\in\mathcal{B}_{ij}$.
Train $t\in\mathcal{T}$ turns around at platform $i$ by travelling
through the crossing-over $(i,j)$, and beginning from platform $j$
starts traversing a different train-line as train $t'\in\mathcal{T}\setminus\{t\}$.
A time window $[\underline{\kappa}_{ij}^{tt'},\overline{\kappa}_{ij}^{tt'}]$
has to be maintained between the departure of the train from platform
$i$ (labelled as train $t$) and arrival at platform $j$ (labelled
as train $t'$). We can write this constraint as follows: 
\begin{align}
\forall(i,j)\in\varphi,\quad\forall(t,t')\in\mathcal{B}_{ij},\quad\underline{\kappa}_{ij}^{tt'}\leq a_{j}^{t'}-d_{i}^{t}\leq\overline{\kappa}_{ij}^{tt'}.\label{eq:turnAroundConstraint}
\end{align}
To clearly illustrate the constraint we consider Figure \ref{crossing-over}.
Here we have two train lines: line 1 and line 2. The terminal platform
on line 1 is platform $i$ and the first platform on line 2 is platform
$j$. The crossing-over from line 1 to line 2 is the arc $(i,j)$.
The train shown in the figure is labelled as $t$ on platform $i$
and labelled as train $t'$ on platform $j$.

\begin{figure}[htp]
\includegraphics{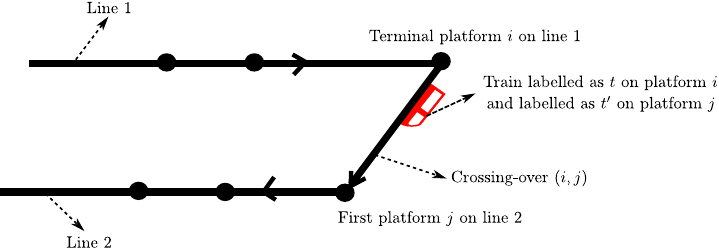} \caption{Trip time constraint associated with a crossing-over}
\label{crossing-over} 
\end{figure}

\subsection{Dwell time constraint}

When any train $t\in\mathcal{T}$ arrives at a platform $i\in\mathcal{N}^{t}$,
it dwells there for a certain time interval denoted by $[\underline{\delta}_{i}^{t},\overline{\delta}_{i}^{t}]$
so that the passengers can get off and get on the train prior to its
departure from platform $j$. The dwell time constraint can be written
as follows: 
\begin{align}
\forall t\in\mathcal{T},\quad\forall i\in\mathcal{N}^{t},\quad\underline{\delta}_{i}^{t}\leq d_{i}^{t}-a_{i}^{t}\leq\overline{\delta}_{i}^{t}.\label{eq:DwellTimeConstraint}
\end{align}

Every train $t\in\mathcal{T}$ arrives at the first platform $\mathcal{N}^{t}(1)$
in its train-path either from the depot or by turning around from
some other line, and departs from the final platform $\mathcal{N}^{t}(|\mathcal{N}^{t}|)$
in order to either return to the depot or start as a new train on
another line by turning around. So, the train $t$ dwells at all platforms
in $\mathcal{N}^{t}$. This is the reason why in Equation \eqref{eq:DwellTimeConstraint}
the platform index $i$ is varied over all elements of the set $\mathcal{N}^{t}$.

\subsection{Connection constraint}

In many cases, a single train connection might not exist between the
origin and the desired destination of a passenger. To circumvent this,
connecting trains are often used at interchange stations. Let $\chi\subseteq\mathcal{N}\times\mathcal{N}$
be the set of all platform pairs situated at the same interchange
stations, where passengers transfer between trains. Let $\mathcal{C}_{ij}$
be the set of connecting train pairs for a platform pair $(i,j)\in\chi$.
For a train pair $(t,t')\in\mathcal{C}_{ij},$ train $t$ is arriving
at platform $i$ and train $t'\in\mathcal{T}$ is departing from platform
$j$. A connection time window denoted by $[\underline{\chi}_{ij}^{tt'},\overline{\chi}_{ij}^{tt'}]$
is maintained between arrival of $t$ and subsequent departure of
$t'$, so that passengers can get off from the first train and get
on the latter. Let $(i,j)\in\chi$. Then the connection constraint
can be written as: 
\begin{align}
 & \forall(i,j)\in\chi,\quad\forall(t,t')\in\mathcal{C}_{ij},\quad\underline{\chi}_{ij}^{tt'}\leq d_{j}^{t'}-a_{i}^{t}\leq\overline{\chi}_{ij}^{tt'}.\label{eq:ConnectionConstraint}
\end{align}

\subsection{Headway constraint}

In any railway network, a minimum amount of time between the departures
and arrivals of consecutive trains on the same track is maintained.
This time is called headway time. For maintaining the quality of passenger
service, many urban railway system keeps an upper bound between the
arrivals and departures of successive trains on the same track, so
that passengers do not have to wait too long before the next train
comes. Let $(i,j)\in\mathcal{A}$ be the track between two platforms
$i$ and $j$, and $\mathcal{H}_{ij}$ be the set of train-pairs who
move along that track successively in order of their departures. Consider
$(t,t')\in\mathcal{H}_{ij}$, and let $[\underline{h}_{i}^{tt'},\overline{h}_{i}^{tt'}]$
and $[\underline{h}_{j}^{tt'},\overline{h}_{j}^{tt'}]$ be the time
windows that have to be maintained between the departures and arrivals
of the trains $t$ and $t'$ from and to the platforms $i$ and $j$
respectively. So, the headway constraint can be written as: 
\begin{align}
 & \forall(i,j)\in\mathcal{A},\quad\forall(t,t')\in\mathcal{H}_{ij},\quad\underline{h}_{i}^{tt'}\leq d_{i}^{t'}-d_{i}^{t}\leq\overline{h}_{i}^{tt'}\;\wedge\ \underline{h}_{j}^{tt'}\leq a_{j}^{t'}-a_{j}^{t}\leq\overline{h}_{j}^{tt'}.\label{eq:SafetyConstraint1}
\end{align}

Similarly, headway times have to be maintained between two consecutive
trains going through a crossing over. Consider any crossing over $(i,j)\in\varphi$
and two such trains, which leave the terminal platform of a train-line
$i$ labelled as $t_{1}$ and $t_{2}$, traverse the crossing over
$(i,j)$, and arrive at platform $j$ of some other train-line labelled
as $t_{1}'$ and $t_{2}'$. The set of all such train quartets $((t_{1},t_{1}'),(t_{2},t_{2}'))$
is represented by $\tilde{\mathcal{H}}_{ij}$. Let $[\underline{h}_{i}^{t_{1}t_{2}},\overline{h}_{i}^{t_{1}t_{2}}]$
be the headway time window between the departures of trains $t_{1}$
and $t_{2}$ from platform $i$ and $[\underline{h}_{j}^{t_{1}'t_{2}'},\overline{h}_{j}^{t_{1}'t_{2}'}]$
be the headway time window between the arrivals of the trains $t_{1}'$
and $t_{2}'$ to the platforms $j$. The associated headway constraints
can be written as: 
\begin{align}
\forall(i,j)\in\varphi,\quad\forall((t_{1},t_{1}'),(t_{2},t_{2}'))\in\tilde{\mathcal{H}}_{ij},\quad\underline{h}_{i}^{t_{1}t_{2}}\leq d_{i}^{t_{2}}-d_{i}^{t_{1}}\leq\overline{h}_{i}^{t_{1}t_{2}}\;\wedge\ \underline{h}_{j}^{t_{1}'t_{2}'}\leq a_{j}^{t_{2}'}-a_{j}^{t_{1}'}\leq\overline{h}_{j}^{t_{1}'t_{2}'}.\label{eq:SafetyConstraint2}
\end{align}

Now we discuss the relation between passenger demand, headway and
number of trains. We denote the passenger demand by $D$, train capacity
by $c$ and utilization rate by $u$. If we denote the number of trains
in service per hour by $n$, then we have 
\[
D=c\times u\times n
\]
\citep{Li2014}. Because the headway time $h$ satisfies the relation
$h=\frac{3600}{n}$, we have 
\begin{equation}
h=\frac{3600\times c\times u}{D}.\label{eq:headway_calculation}
\end{equation}
It should be noted that the train capacity $c$ and the utilization
rate $u$ are constant parameters. However the passenger demand varies
with time. As a result, in the equation above trains will have different
headway at different periods.

\subsection{Total travel time constraint}

The train-path of a train is the directed path containing all platforms
and tracks visited by it in chronological order. To maintain the quality
of service in the railway network, for every train $t\in\mathcal{T}$,
the total travel time to traverse its train-path has to stay within
a time window $[\underline{\tau}_{\mathcal{P}}^{t},\overline{\tau}_{\mathcal{P}}^{t}]$.
We can write this constraint as follows: 
\begin{align}
\forall t\in\mathcal{T},\quad\underline{\tau}_{\mathcal{P}}^{t}\leq a_{\mathcal{N}^{t}(|\mathcal{N}^{t}|)}^{t}-d_{\mathcal{N}^{t}(1)}^{t}\leq\overline{\tau}_{\mathcal{P}}^{t},\label{eq:TotalTravelTimeConstraints}
\end{align}
where $\mathcal{N}^{t}(1)$ and $\mathcal{N}^{t}(|\mathcal{N}^{t}|)$
are the first and last platform in the train-path of $t$.

\subsection{Domain of the event times}

Without any loss of generality, we set the time of the first event
of the railway service period, which corresponds to the departure
of the first train of the day from some platform, to start at zero
second. By setting all trip times and dwell times to their maximum
possible values we can obtain an upper bound for the final event of
the railway service period, which is the arrival of the last train
of the day at some platform, denoted by $m\in\integers_{++}$. So
the domain of the decision variables can be expressed by the following
equation: 
\begin{align}
\forall t\in\mathcal{T},\quad\forall i\in\mathcal{N}^{t},\quad0\leq a_{i}^{t}\leq m,0\leq d_{i}^{t}\leq m.\label{eq:domain}
\end{align}

In vector notation the decision variables are denoted by $a=\left((a_{i}^{t})_{i\in\mathcal{N}^{t}}\right)_{t\in\mathcal{T}}$
and $d=\left((d_{i}^{t})_{i\in\mathcal{N}^{t}}\right)_{t\in\mathcal{T}}$.

\section{First optimization model}

\label{Preli_Opt} In this section we formulate the first optimization
model that minimizes the total energy consumed by all trains in the
railway network. The organization of this section is as follows. First,
in order to keep the proofs less cluttered, we introduce an equivalent
constraint graph notation. Then we formulate and justify the first
optimization problem. Finally we show that the nonlinear objective
of the initial optimization model can be approximated as a linear
one by applying least-squares. This results in a linear optimization
problem, which has the interesting property that its optimal solution
is attained by an integral vector. 

\subsection{Constraint graph notation}

Each of the constraints described by Equations \eqref{eq:TripTimeConstraint}-\eqref{eq:TotalTravelTimeConstraints}
is associated with two event times (either arrival or departure time
of trains at stations), where one of them precedes another by a time
difference dictated by the time window of that constraint. This observation
helps us to convert our initial notation into an equivalent constraint
graph notation which we describe as follows. 

\paragraph{Converting the initial notation into an equivalent constraint graph}
\begin{itemize}
\item \emph{Nodes of the constraint graph:} All event times in the original
notation are treated as nodes in the constraint graph, the set of
those nodes is denoted by $\bar{\mathcal{N}}$ and the value associated
with a node $i\in\bar{\mathcal{N}}$ is denoted by $x_{i}$, which
represents the arrival or departure time of some train from a platform.
Consider any two nodes in the constraint graph; if there exists a
constraint between the two in the original notation, then in the constraint
graph we create a directed arc between them, the start node being
the first event and the end node being the later one. The set of arcs
thus created in the constraint graph is denoted by $\bar{\mathcal{A}}$.
Note that there cannot be more than one arc between two nodes in the
constraint graph. 
\item \emph{Arcs of the constraint graph:} With each arc $(i,j)\in\bar{\mathcal{A}}$
we associate a time window $[l_{ij},u_{ij}]$ with their values determined
from the Equations \eqref{eq:TripTimeConstraint}-\eqref{eq:TotalTravelTimeConstraints}.
So, each arc $(i,j)\in\bar{\mathcal{A}}$ corresponds to a constraint
of the form $l_{ij}\leq x_{j}-x_{i}\leq u_{ij}$. The set of all arcs
associated with trip time constraints is expressed by $\bar{\mathcal{A}}_{\textrm{trip}}\subset\bar{\mathcal{A}}$. 
\end{itemize}
\begin{figure}[h]
\includegraphics[scale=0.5]{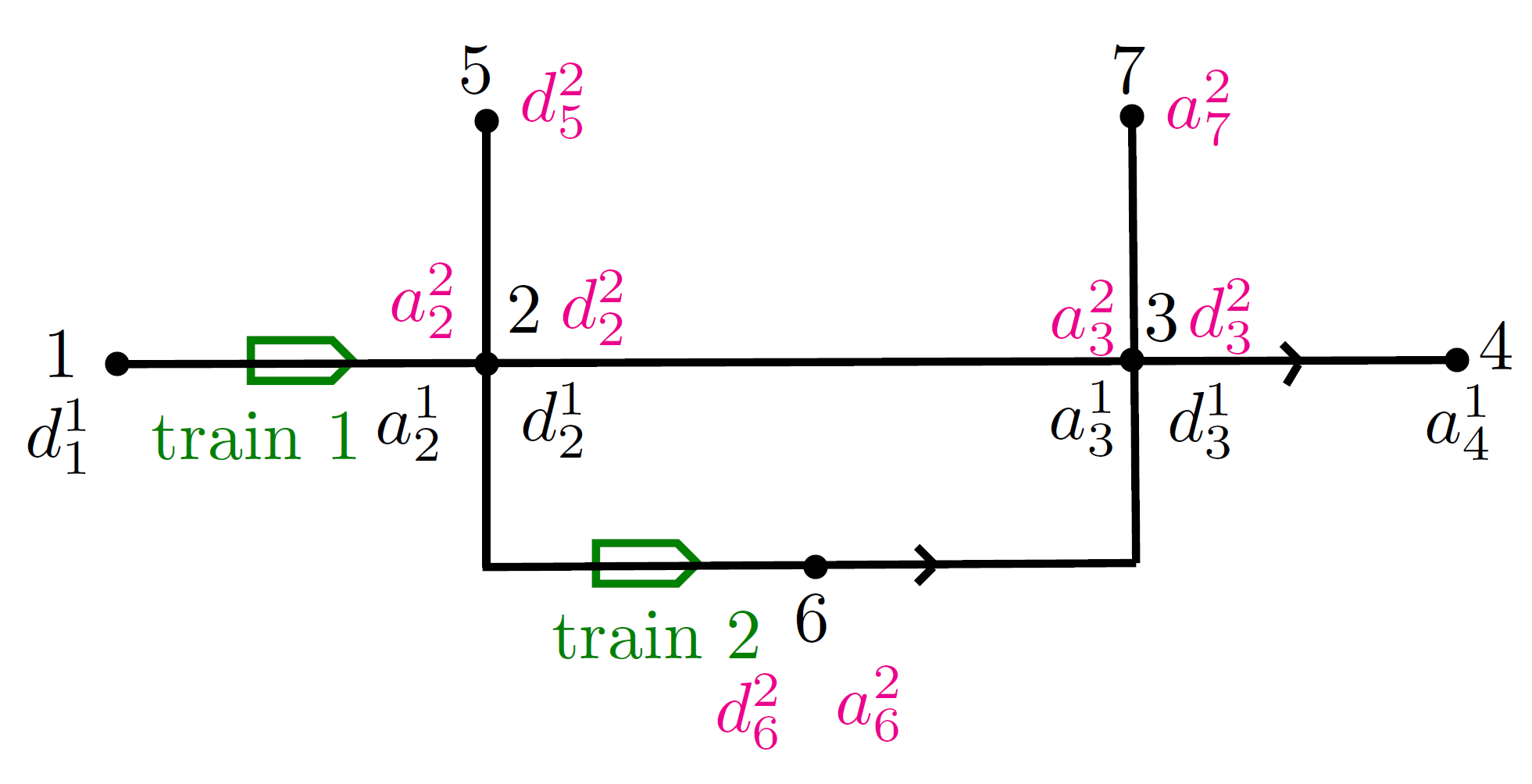}

\caption{Example of a very simple railway network. \label{example1}}
\end{figure}

\begin{figure}[h]
\includegraphics[scale=0.3]{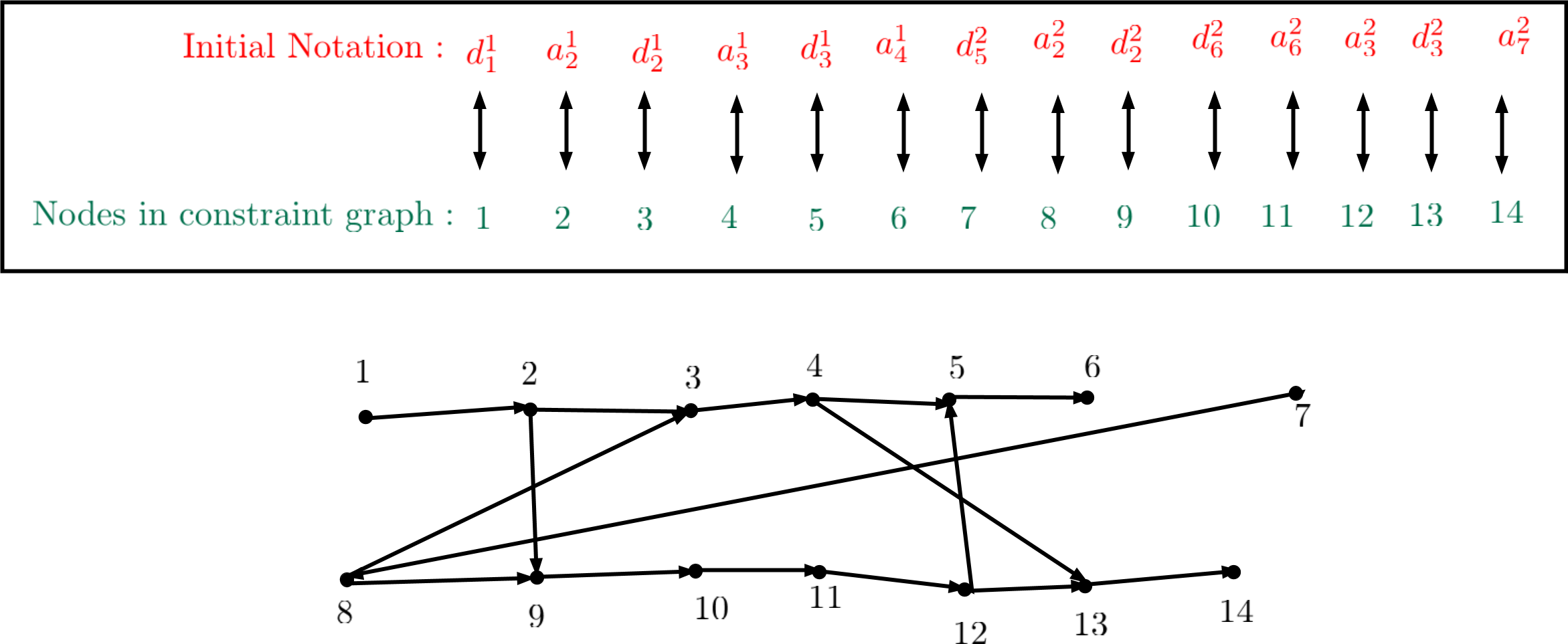}

\caption{The constraint graph for the railway network of Figure~\ref{example1}.}

\label{fig:Constraint Graph} 
\end{figure}

\paragraph{Example}

Figure~\ref{example1} represents a very simple network with 7 stations
represented by black dots and 2 trains represented by the directed
rectangles. The pointed edge of the symbol indicates the direction.
In this network, $\mathcal{T}={\color{black}\{1,2\}}$ and and the
stations are enumerated as $\{1,2,3,4,5,6,7\}$. There are two train
lines as shown in the figure. It should be noted that nodes represent
stations, not platfrms. Nodes 2 and 3 represent interchange stations,
where there are two platforms for both lines on different levels.
So, the event times at node 2 and node 3 are associated with different
platforms and are differentiated using black and magenta colors. The
set of tracks visited by train $1$ is $\mathcal{A}^{{\color{black}1}}=\{(1,2),(2,3),(3,4)\}$,
and the set of tracks visited by train $2$ is $\mathcal{A}^{{\color{black}2}}=\{(5,2),(2,6),(6,3),(3,7)\}$.
The set of all tracks is then $\mathcal{A}=\mathcal{A}^{{\color{black}1}}\cup\mathcal{A}^{{\color{black}2}}$.
The event times corresponding to train $1$ and train $2$ are shown
in black and magenta colours respectively in Figure~\ref{example1}.
Applying the conversion process described above we can convert the
initial notation in Figure~\ref{example1} to the constraint graph
shown in Figure~\ref{fig:Constraint Graph}. .

\subsection{Formulation of the first optimization model}

A train consumes most of its required electrical energy during the
acceleration phase of making a trip from an origin platform to a destination
platform. Trip time constraints play the most important role in energy
consumption and regenerative energy production of trains. Once the
trip time for a trip is fixed, an energy optimal speed profile can
be calculated efficiently by existing software \citep{selTrac}, \citep[page 285]{Howlett1995},
such as Thales \textbf{T}rain \textbf{K}inetics, \textbf{D}ynamics
and \textbf{C}ontrol (TKDC) Simulator in our case. The TKDC simulator
assumes maximum accelerating - speed holding - coasting - maximum
braking strategy for calculation of speed profile. Theoretically this
is the optimal speed \textcolor{black}{profile according to the monograph
\citep{Howlett1995}. For calculation of the optimal speed profile
of a train while making a trip, we refer the interested reader to
the highly cited papers \citep{Jiaxin1993,Howlett2000,Howlett2009,Khmelnitsky2000,Liu2003}.
An excellent state-of-the art review of calculating energy-efficient
or energy-optimal speed profile can be found in \citep{Albrecht2015}.
The key assumptions that we have made in calculating the speed profiles,
all of which are well-established in relevant literature \citep{Albrecht2015},
are as follows.}
\begin{itemize}
\item \textcolor{black}{The acceleration and braking control forces are
assumed to be uniformly bounded or bounded by magnitude constraints.
These magnitude constraints are assumed to be dependent on the speed
of the train and non-increasing.}
\item \textcolor{black}{The initial speed of the train from the origin platform
and final speed at the destination platform are assumed to be zero.
The speed is assumed to be strictly positive at every other position.}
\item \textcolor{black}{Only the force associated with positive acceleration
consumes energy. }
\item \textcolor{black}{The rate of energy dissipation from frictional resistance
is assumed to be a quadratic function of the form $a_{0}+a_{1}v+a_{2}v^{2}$
(also known as the }\textcolor{black}{\emph{Davis formula}}\textcolor{black}{{}
\citep{Davis1926}), where $v$ stands for speed of the train and
$a_{0},a_{1},a_{2}$ are non-negative physical constants. These constants
can be determined approximately in practice by measuring the difference
between the applied tractive force and the net tractive force. }
\item \textcolor{black}{We have assumed the position of the train to be
an independent variable and speed of the train to be a dependent variable.
This assumption appears all the major contributions. This assumption
appears all the major contributions. A notable exception is \citep{Khmelnitsky2000},
which considers kinetic energy and time as the dependent variables
and position as the independent variable.}
\end{itemize}
The electrical power consumption and regeneration of a train on a
track is determined by its speed profile, so the optimal speed profile
also gives the power versus time graph (\emph{power graph} in short)
for that trip. However, in the total railway service period there
are many active trains, whose movements are coupled by the associated
constraints. So, finding the energy-minimal trip time for a single
trip in an isolated manner can result in a infeasible timetable. Consider
an arc $(i,j)\in\bar{\mathcal{A}}_{\textrm{trip}}$ in the constraint
graph, associated with some trip time constraint. Let us denote the
energy consumption function for that trip $f_{ij}:\reals_{++}\rightarrow\reals_{++}$
with argument $(x_{j}-x_{i})$. The first optimization problem with
the objective to minimize the total energy consumption of the trains
can be written as: 
\begin{equation}
\begin{array}{ll}
\mbox{minimize} & \sum_{(i,j)\in\bar{\mathcal{A}}_{\textrm{trip}}}f_{ij}(x_{j}-x_{i})\\
\mbox{subject to} & l_{ij}\leq x_{j}-x_{i}\leq u_{ij},\qquad\forall(i,j)\in\bar{\mathcal{A}}\\
 & 0\leq x_{i}\leq m,\qquad\forall i\in\bar{\mathcal{N}},
\end{array}\label{preOpt1}
\end{equation}
where the decision vector is $(x_{i})_{i\in\bar{\mathcal{N}}}\in\reals^{|\bar{\mathcal{N}}|}$.

The exact analytical form of every component of the objective function,
\emph{i.e.}, $f_{ij}(x_{j}-x_{i})$ for $(i,j)\in\bar{A}_{\textrm{trip}}$
is not known and may be intractable \citep{Howlett2009}. However,
irrespective of the exact analytical form, the energy function can
be shown to be monotonically decreasing in trip time, \emph{i.e.},
it is \emph{non-increasing} with the increase in trip time, if the
optimal speed profile is followed \citep{Milroy1980}. Even when a
train is manually driven with possibly suboptimal driving strategies,
the average energy consumption of the train is found empirically to
be monotonically decreasing in the trip time \citep{Pena-Alcaraz2012a}.

Also, the energy function is relatively easy to measure in practice
\citep[Section 1.5]{Howlett1995}. For any $(i,j)\in\bar{A}_{\textrm{trip}}$,
we denote the measured trip times $(x_{j}^{(1)}-x_{i}^{(1)}),\ldots,(x_{j}^{(p)}-x_{i}^{(p)})$
and the corresponding energy consumption data $f_{ij}^{(1)},\ldots,f_{ij}^{(p)}$.

In any subway system, the amount by which the trip time is allowed
to vary in Equations \eqref{eq:TripTimeConstraint} and \eqref{eq:turnAroundConstraint}
is typically on the order of seconds \citep{Liebchen2007}, which
motivates us to make the following assumption.

\begin{assumption} \label{assum:triptime}
 The amount by which the trip time is allowed to vary is on the order of seconds, i.e., for any  the trip time window  is on the order of seconds.
\end{assumption}

The monotonically decreasing nature of the energy function together
with Assumption \ref{assum:triptime} allows us to approximate the
energy function $f_{ij}(x_{j}-x_{i})$ as an affine function. Recall
that in practice, we can measure the energy $f_{ij}^{(1)},\ldots,f_{ij}^{(p)}$
and associated trip times $(x_{j}^{(1)}-x_{i}^{(1)}),\ldots,(x_{j}^{(p)}-x_{i}^{(p)})$,
which is obtainable easily with present technology \citep[Section 1.5]{Howlett1995}.

Now we want to formulate an optimization problem which will provide
us with the best possible affine approximation of the energy function
$f_{ij}(x_{j}-x_{i})$. We do so by applying least-squares and fit
a straight line through measured energy versus trip time data. We
seek an affine function $c_{ij}(x_{j}-x_{i})+b_{ij}=(x_{j}-x_{i},1)^{T}(c_{ij},b_{ij})$
where we want to determine $c_{ij}$ and $b_{ij}$. 

The affine function approximates the measured energy in the least-squares
sense as follows: 
\begin{align}
(c_{ij},b_{ij})= & \argmin_{(\tilde{c}_{ij},\tilde{b}_{ij})}\sum_{k=1}^{p}\left(\tilde{c}_{ij}(x_{j}^{(k)}-x_{i}^{(k)})+\tilde{b}_{ij}-f_{ij}^{(k)}\right)^{2}\nonumber \\
= & \argmin_{(\tilde{c}_{ij},\tilde{b}_{ij})}\left\Vert \begin{bmatrix}(x_{j}^{(1)}-x_{i}^{(1)},1)^{T}\\
\vdots\\
(x_{j}^{(p)}-x_{i}^{(p)},1)^{T}
\end{bmatrix}\begin{bmatrix}\tilde{c}_{ij}\\
\tilde{b}_{ij}
\end{bmatrix}-\begin{bmatrix}f_{ij}^{(1)}\\
\vdots\\
f_{ij}^{(p)}
\end{bmatrix}\right\Vert _{2}^{2}
\end{align}
The problem above is an unconstrained optimization problem with convex
quadratic differentiable objective. So it can be solved by taking
the gradient with respect to $(\tilde{c}_{ij},\tilde{b}_{ij})$, setting
the result equal to zero vector and then solving for ${(\tilde{c}_{ij},\tilde{b}_{ij})}$.
This yields the following closed form solution: 
\begin{align}
\begin{bmatrix}c_{ij}\\
b_{ij}
\end{bmatrix}=\left(\begin{bmatrix}(x_{j}^{(1)}-x_{i}^{(1)},1)^{T}\\
\vdots\\
(x_{j}^{(p)}-x_{i}^{(p)},1)^{T}
\end{bmatrix}^{T}\begin{bmatrix}(x_{j}^{(1)}-x_{i}^{(1)},1)^{T}\\
\vdots\\
(x_{j}^{(p)}-x_{i}^{(p)},1)^{T}
\end{bmatrix}\right)^{-1}\begin{bmatrix}(x_{j}^{(1)}-x_{i}^{(1)},1)^{T}\\
\vdots\\
(x_{j}^{(p)}-x_{i}^{(p)},1)^{T}
\end{bmatrix}^{T}\begin{bmatrix}f_{ij}^{(1)}\\
\vdots\\
f_{ij}^{(p)}
\end{bmatrix}\label{eq: Least Square Solution}
\end{align}

Using Equation~\eqref{eq: Least Square Solution}, we can approximate
the nonlinear objective of the optimization problem \eqref{preOpt1}
as an affine one: $\sum_{(i,j)\in\bar{\mathcal{A}}_{\textrm{trip}}}c_{ij}(x_{i}-x_{j})+b_{ij}$.
A measurement of the quality of such fittings is given by the \emph{coefficient
of determination}, which can vary between 0 to 1, with 0 being the
worst and 1 being the best \citep[page 518]{Kapadia2005}. In our
numerical studies the mean coefficient of determination of the energy
fittings over all the different trips of all the trains is found to
be 0.9483 with a standard deviation of 0.05, which justifies our approach.
We can also discard the $b_{ij}$s from the objective, as it has no
impact on the minimizer. Thus we arrive at the following linear optimization
problem to minimize the total energy consumption of the trains: 
\begin{equation}
\begin{array}{ll}
\textnormal{minimize} & \sum_{(i,j)\in\bar{\mathcal{A}}_{\textrm{trip}}}c_{ij}(x_{j}-x_{i})\\
\textnormal{subject to} & l_{ij}\leq x_{j}-x_{i}\leq u_{ij},\qquad\forall(i,j)\in\bar{\mathcal{A}}\\
 & 0\leq x_{i}\leq m,\qquad\forall i\in\bar{\mathcal{N}}.
\end{array}\label{preOpt2}
\end{equation}

Note that, we have not used the same cost-time curve for all the trips.
Each of the constituent parts $c_{ij}(x_{j}-x_{i})$ of the objective
function $\sum_{(i,j)\in\bar{\mathcal{A}}_{\textrm{trip}}}c_{ij}(x_{j}-x_{i})$
in Problem (12) represents approximated affine function for each of
the trips considered in the optimization problem. If optimal speed
profiles for the trips are available to the railway management, from
the first optimization model the optimal trip times for those optimal
speed profiles can be found. If available speed profiles are suboptimal,
then the first optimization model would still produce an energy-efficient
timetable with best trip times subject to the available speed profiles.

An important property of this optimization model is that the polyhedron
associated with optimization problem has only integer vertices, so
the optimal value is attained by an integral vector. A necessary and
sufficient condition of integrality of the vertices of a polyhedron
is given by the following theorem \citep[page 269, Theorem 19.3]{Schrijver1998},
which we will use to prove the subsequent proposition. 
\begin{thm}
\label{IntegralityTheorem} Let $A$ be a matrix with entries $0,+1$,
or $-1$. For all integral vectors $a,b,c,d$ the polyhedron $\{x\in\reals^{n}\mid c\preceq x\preceq d,a\preceq Ax\preceq b\}$
has only integral vertices if and only if for each nonempty collection
of columns of $A$, denoted by $C$, there exist two subsets , $C_{1}$
and $C_{2}$ such that $C_{1}\cup C_{2}=C,C_{1}\cap C_{2}=\emptyset$,
and the sum of the columns in $C_{1}$ minus the sum of the columns
in $C_{2}$ is a vector with entries $0,1$ and $-1$. 
\end{thm}

\begin{prop}
The optimization problem \eqref{preOpt2} has an integral optimal
solution. 
\end{prop}
\begin{proof}
We write the problem \eqref{preOpt2} in vector form. We construct
a cost vector $c$, such that a component of that vector is $c_{ij}$
if it is associated with a trip time constraint in the original notation,
and zero otherwise. Construct integral vectors $l=(l_{ij})_{(i,j)\in\bar{\mathcal{A}}}$,
$u=(u_{ij})_{(i,j)\in\bar{\mathcal{A}}}$ and matrix $A\in\{-1,0,1\}^{|\bar{\mathcal{A}}|\times|\bar{\mathcal{N}}|}$
such that the $(k,i)$th entry of the matrix $A$, denoted by $a_{ki}$,
is associated with the $k$th hyperarc and $i$th node of the constraint
graph as follows: 
\[
a_{ki}=\begin{cases}
1\quad\text{if node \ensuremath{i} is the end node of hyperarc k,}\\
-1\quad\text{if node \ensuremath{i} is the start node of hyperarc k,}\\
0\quad\text{otherwise.}
\end{cases}
\]
So, the vector form of the optimization problem \eqref{preOpt2} is:
\begin{equation}
\begin{array}{ll}
\mbox{minimize} & c^{T}x\\
\mbox{subject to} & l\preceq Ax\preceq u,\\
 & 0\preceq x\preceq m\ones.
\end{array}\label{preOptVector}
\end{equation}
Consider any nonempty collection of columns of $A$ denoted by $C$.
Take $C_{1}=C$ and $C_{2}=\emptyset$. Then the sum of the columns
in $C_{1}$ minus the sum of the columns in $C_{2}$ will be a vector
with entries $0,1$ and $-1$, because in $A$ there cannot exist
more than one row corresponding to an arc between two nodes of the
constraint graph and each such row has exactly two nonzero entries,
a $+1$ and a $-1$. So, by Theorem~\ref{IntegralityTheorem} the
polyhedron $\{x\in\reals^{|\bar{\mathcal{N}}|}:l\preceq Ax\preceq u,0\preceq x\preceq m\ones\}$
has only integral vertices and optimizing the linear objective in
problem~\eqref{preOptVector} over this polyhedron will result in
an integral solution.
\end{proof}
After solving the linear programming problem \eqref{preOpt2}, we
obtain an integral timetable, which we will call the energy minimizing
timetable (\textbf{EMT}). We denote the optimal decision vector of
this timetable by $\bar{x}$ in the constraint graph notation and
$\left((\bar{a}_{i}^{t},\bar{d}_{i}^{t})_{i\in\mathcal{N}}\right)_{t\in\mathcal{T}}$
in the original notation.

\section{Second optimization model}

\label{Final_Opt} In this section we modify the trip time constraints
such that the total energy consumption of the final timetable is kept
at the same minimum as the EMT. Then, we describe our optimization
strategy aimed to maximize the utilization of regenerative energy
of braking trains, and we present the second optimization model.

\subsection{Keeping the total energy consumption at minimum}

In any feasible timetable, if the trip times are kept to be the same
as the ones obtained from the EMT, then the energy optimal speed profiles
for all trains will be the same. As a result, the energy consumption
associated with that timetable will remain at the same minimum as
found in the EMT. So, in the second optimization problem, instead
of using the trip time constraint, for every trip we fix the trip
time to the value in the EMT, \emph{i.e.,} 
\begin{align}
\forall t\in\mathcal{T},\quad\forall(i,j)\in\mathcal{A}^{t},\quad a_{j}^{t}-d_{i}^{t}=\bar{a}_{j}^{t}-\bar{d}_{i}^{t},\label{eq:TripTimeConstraintSynModel}
\end{align}
and 
\begin{align}
\forall(i,j)\in\varphi,\quad\forall(t,t')\in\mathcal{B}_{ij},\quad a_{j}^{t'}-d_{i}^{t}=\bar{a}_{j}^{t'}-\bar{d}_{i}^{t}.\label{eq:TurnAroundConstraintSyncModel}
\end{align}
For all other constraints, bounds are allowed to vary as described
by Equations \eqref{eq:DwellTimeConstraint}-\eqref{eq:TotalTravelTimeConstraints}.
As a consequence of fixing all trip times, the power graph of every
trip made by any train becomes known to us, since it depends on the
corresponding optimal speed profile calculated in real-time by existing
software \citep{selTrac}, \citep[page 285]{Howlett1995}.

\subsection{Maximizing the utilization of regenerative energy of braking trains}

In this subsection we describe our strategy to maximize the utilization
of the regenerative energy produced by the braking trains. Strategies
based on transfer of regenerative braking energy back to the electrical
grid requires specialized technology such as reversible electrical
substations \citep{t2kreport}. A strategy based on storing is not
feasible with present technology, because storage options such as
super-capacitors, fly-wheels, \emph{e.t.c.}, have drastic discharge
rates besides being too expensive \citep[page 66]{nano2014}, \citep[page 92]{Droste-Franke2012}.
A better strategy that can be used with existing technology \citep{bombardier}
is to transfer the regenerative energy of a braking train to a nearby
and simultaneously accelerating train, if both of them operate under
the same electrical substation. We call such pairs of trains \emph{suitable
train pairs}. So our objective is to maximize the total overlapped
area between the graphs of power consumption and regeneration of all
suitable train pairs. To model this mathematically, we are faced with
the following tasks: \emph{i)} define suitable train pairs, \emph{ii)}
provide a tractable description of the overlapped area between power
graphs of such a pair. We describe them as follows. 
\begin{figure}[h]
\includegraphics{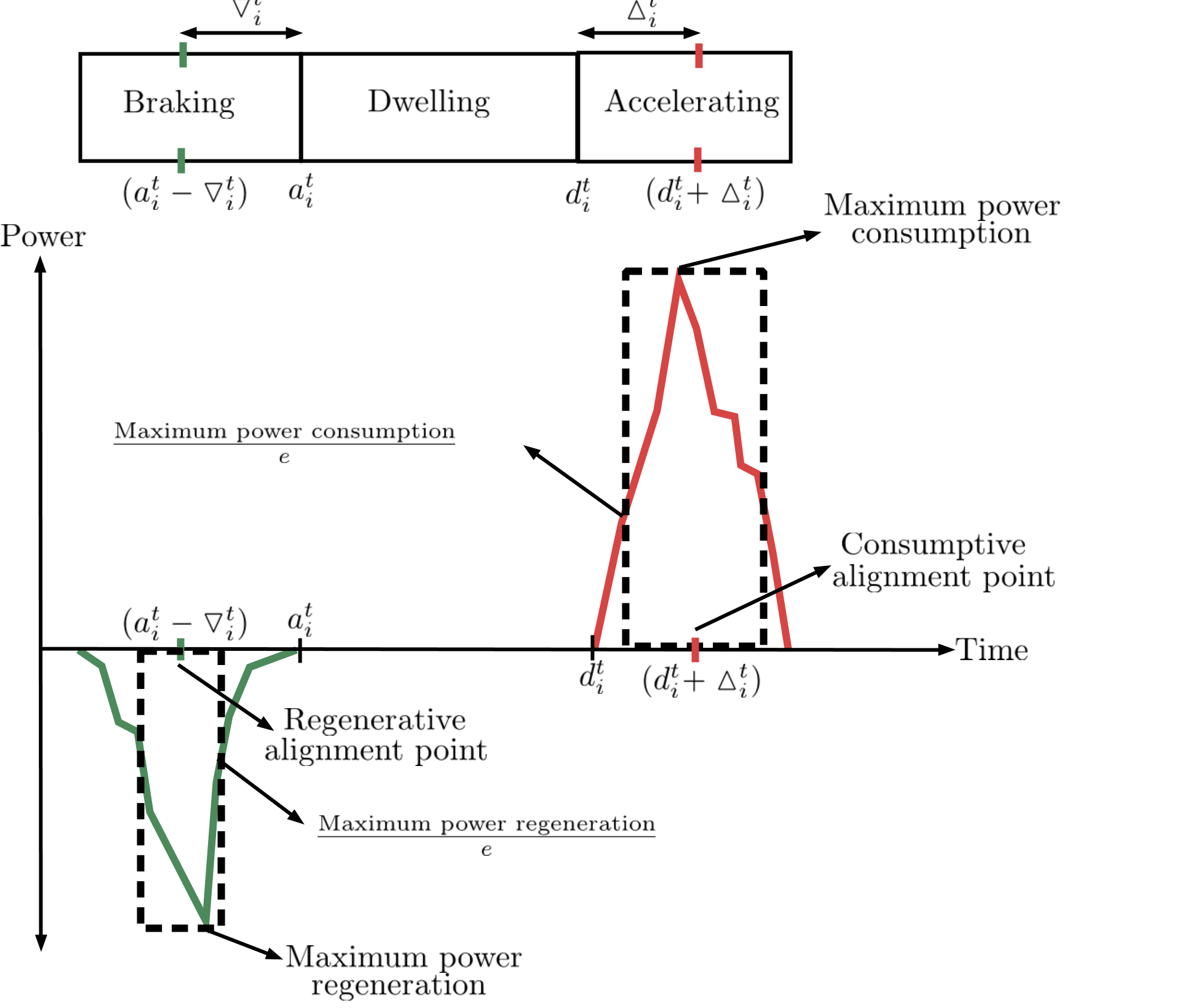} \caption{Applying $\frac{1}{e}$ heuristic to power graphs }
\label{Notation} 
\end{figure}

\subsubsection{Defining suitable train pairs}

We consider platform pairs who are opposite to each other and are
powered by the same electrical substations. Thus, the transmission
loss in transferring electrical energy between them is negligible.
In our work we The set containing all such platform pairs is denoted
by $\Omega$. Consider any such platform pair $(i,j)\in\Omega$, and
let $\mathcal{T}_{i}\subseteq\mathcal{T}$ be the set of all trains
which arrive at, dwell and then depart from platform $i$. Suppose,
$t\in\mathcal{T}_{i}$. Now, we are interested in finding another
train $\tilde{t}$ on platform $j$, \emph{i.e.}, $\tilde{t}\in\mathcal{T}_{j}$,
which along with $t$ would form a suitable pair for the transfer
of regenerative braking energy. To achieve this, we use the EMT. Among
all trains going through platform $j$, the one which is temporally
closest to $t$ in the energy-minimizing timetable is be the best
candidate to form a pair with $t$. The temporal proximity can be
of two types with respect to $t$, which results in the following
definitions.
\begin{defn}
Consider any $(i,j)\in\Omega$. For every train $t\in\mathcal{T}_{i}$,
the train $\overset{\rightharpoonup}{t}\in\mathcal{T}_{j}$ is called
the \textbf{temporally closest train to the right of $t$} if 
\begin{align}
\overset{\rightharpoonup}{t}=\underset{t'\in\{x\in\mathcal{T}_{j}:0\leq\frac{\bar{a}_{j}^{x}+\bar{d}_{j}^{x}}{2}-\frac{\bar{a}_{i}^{t}+\bar{d}_{i}^{t}}{2}\leq r\}}{\textnormal{argmin}}\left\{ \left|\frac{\bar{a}_{i}^{t}+\bar{d}_{i}^{t}}{2}-\frac{\bar{a}_{j}^{t'}+\bar{d}_{j}^{t'}}{2}\right|\right\} ,\label{eq:temporallyClosestTrainRight}
\end{align}
where $r$ is an empirical parameter determined by the timetable designer
and is much smaller than the time horizon of the entire timetable. 
\end{defn}

\begin{defn}
Consider any $(i,j)\in\Omega$. For every train $t\in\mathcal{T}_{i}$,
the train $\overset{\leftharpoonup}{t}\in\mathcal{T}_{j}$ is called
the \textbf{temporally closest train to the left of $t$} if 
\begin{align}
\overset{\leftharpoonup}{t}=\underset{t'\in\{x\in\mathcal{T}_{j}:0<\frac{\bar{a}_{i}^{t}+\bar{d}_{i}^{t}}{2}-\frac{\bar{a}_{j}^{x}+\bar{d}_{j}^{x}}{2}\leq r\}}{\textnormal{argmin}}\left\{ \left|\frac{\bar{a}_{i}^{t}+\bar{d}_{i}^{t}}{2}-\frac{\bar{a}_{j}^{t'}+\bar{d}_{j}^{t'}}{2}\right|\right\} .\label{eq:temporallyClosestTrainLeft}
\end{align}
\end{defn}

\begin{defn}
Consider any $(i,j)\in\Omega$. For every train $t\in\mathcal{T}_{i}$,
the train $\tilde{t}\in\mathcal{T}_{j}$ is called the \textbf{temporally
closest train to $t$} if 
\begin{align}
\tilde{t}=\underset{t'\in\{\overset{\rightharpoonup}{t},\overset{\leftharpoonup}{t}\}}{\textnormal{argmin}}\left\{ \left|\frac{\bar{a}_{i}^{t}+\bar{d}_{i}^{t}}{2}-\frac{\bar{a}_{j}^{t'}+\bar{d}_{j}^{t'}}{2}\right|\right\} .\label{eq:temporallyClosestTrain}
\end{align}
If both $\overset{\rightharpoonup}{t}$ and $\overset{\leftharpoonup}{t}$
are temporally equidistant from $t$, we pick one of them arbitrarily. 
\end{defn}
Any synchronization process between a suitable train pair (SPSTP)
can be described by specifying the corresponding $i$, $j$, $t$
and $\tilde{t}$ by using the definitions above. We construct a set
of all SPSTPs, which we denote by $\mathcal{E}$. Each element of
this set is a tuple of the form $(i,j,t,\tilde{t})$. Because $\tilde{t}$
is unique for any $t$ in each element of $\mathcal{E}$, we can partition
$\mathcal{E}$ into two sets denoted by $\overset{\rightharpoonup}{\mathcal{E}}$
and $\overset{\leftharpoonup}{\mathcal{E}}$ containing elements of
the form $(i,j,t,\overset{\rightharpoonup}{t})$ and $(i,j,t,\overset{\leftharpoonup}{t})$
respectively.
\begin{rem*}
\textcolor{black}{Now we compare our design choice to that of \citep{Yang2013},
which considers two additional scenarios: two trains 1) on two consecutive
platforms on the same track, or 2) on the same platform on same track.
In the first case, there will of course be a significant distance
between the two trains, and in the second case, we need to maintain
at least the safety distance between the two trains irrespective of
how small the headway is. First, to transfer the regenerative energy
to the accelerating trains, we would need either supercapacitors,
or fly-wheels, both of which have drastic discharge rates \citep[page 66]{nano2014},
\citep[page 92]{Droste-Franke2012}. Additionally, if we consider
resistive transmission loss, in practice there would be very little
or no energy transfer due to the power loss along the line. Note that
in \citep{Yang2013} this issue would not arise, as they have assumed
no transmission loss.}
\end{rem*}

\subsubsection{Description of the overlapped area between power graphs}

The power graph during accelerating and braking is highly nonlinear
in nature with no analytic form, as shown in Figure~\ref{Notation}.
So, maximizing the exact overlapped area will lead to an intractable
optimization problem. However, the existence of dominant peaks with
sharp falls allows us to apply a robust lumping method such as $\frac{1}{e}$
heuristic \citep[page 33-34]{Mahajan2008} to approximate the power
graphs as rectangles. The $\frac{1}{e}$ heuristic is applied as follows
(see Figure~\ref{Notation}). The height of the rectangle is the
maximum power, and the width is the interval with extreme points corresponding
to power dropped at $1/e$ of the maximum. For the sharp drop from
the peak, such rectangles are very robust approximations to the original
power graph containing the most concentrated part of the energy, \emph{e.g.,},
if the drop were exponential, then the energy contained by the rectangle
would have been exactly equal to that of the original curve \citep[page 33-34]{Mahajan2008}.
After converting both the power graphs to rectangles, maximizing the
overlapped area under those rectangles is equivalent to aligning the
midpoint of the width of the rectangles; we call such a midpoint \textbf{regenerative
or consumptive alignment point}. These alignment points act as virtual
peaks of the approximated power graphs. As shown in Figure~\ref{Notation},
for a train $t$ in its braking phase prior to its arrival at platform
$i$, the relative distance of $a_{i}^{t}$ from the regenerative
alignment point is denoted by $\triangledown_{i}^{t}$, while during
acceleration the relative distance of the consumptive alignment point
from $d_{i}^{t}$ is denoted by $\vartriangle_{j}^{t}$. Note that
both relative distances are known parameters for the current optimization
problem.

\subsection{Second optimization model}

Consider an element $(i,j,t,\overset{\rightharpoonup}{t})\in\overset{\rightharpoonup}{\mathcal{E}}$.
To ensure the transfer of maximum possible regenerative energy from
the braking train $\overset{\rightharpoonup}{t}$ to the accelerating
train $t$, we aim to align both their alignment points such that
$d_{i}^{t}+\vartriangle_{i}^{t}=a_{j}^{\overset{\rightharpoonup}{t}}-\triangledown_{j}^{\overset{\rightharpoonup}{t}},$
or keep them as close as possible otherwise. Similarly, for any $(i,j,t,\overset{\leftharpoonup}{t})\in\overset{\leftharpoonup}{\mathcal{E}}$,
our objective is $d_{j}^{\overset{\leftharpoonup}{t}}+\vartriangle_{j}^{\overset{\leftharpoonup}{t}}=a_{i}^{t}-\triangledown_{i}^{t}$
, or as close as possible. Let a decision vector $y$ be defined as
\begin{align}
y=\Big(\big(d_{i}^{t}+\vartriangle_{i}^{t}-a_{j}^{\overset{\rightharpoonup}{t}}+\triangledown_{j}^{\overset{\rightharpoonup}{t}}\big)_{(i,j,t,\overset{\rightharpoonup}{t})\in\overset{\rightharpoonup}{\mathcal{E}}},\big(d_{j}^{\overset{\leftharpoonup}{t}}+\vartriangle_{j}^{\overset{\leftharpoonup}{t}}-a_{i}^{t}+\triangledown_{i}^{t}\big)_{(i,j,t,\overset{\leftharpoonup}{t})\in\overset{\leftharpoonup}{\mathcal{E}}}\Big).\label{eq:defOfy}
\end{align}
Then our goal comprises of two parts: $1$) maximize the number of
zero components of $y$ which corresponds to minimizing $\card(y)$,
and $2$) keep the nonzero components as close to zero as possible
which corresponds to minimizing the $\ell_{1}$ norm of $y$, $\|y\|_{1}$.
Combining these two we can write the exact optimization problem as
follows: 
\begin{equation}
\begin{array}{ll}
\mbox{minimize}\quad\card(y)+\gamma\|y\|_{1}\\
\mbox{subject to}\\
\text{Equations}\eqref{eq:DwellTimeConstraint}-\eqref{eq:TotalTravelTimeConstraints},\eqref{eq:TripTimeConstraintSynModel},\eqref{eq:TurnAroundConstraintSyncModel},\eqref{eq:defOfy},\\
0\leq a_{i}^{t}\leq m,0\leq d_{i}^{t}\leq m,\quad\forall i\in\mathcal{N}^{t},\quad\forall t\in\mathcal{T},
\end{array}\label{syncNonconvexForm}
\end{equation}
where $\gamma$ is a positive weight, and decision variables are $a$,
$d$ and $y$. The objective function is nonconvex as shown next.
Take the convex combination of the vectors $2e_{1}/\gamma$ and $0$
with convex coefficients $1/2$. Then, 
\[
\card\left(\frac{e_{1}}{\gamma}\right)+\gamma\left\Vert \frac{e_{1}}{\gamma}\right\Vert _{1}=2>\frac{1}{2}\left(\card\left(\frac{2e_{1}}{\gamma}\right)+\gamma\left\Vert \frac{2e_{1}}{\gamma}\right\Vert _{1}\right)+\frac{1}{2}\left(\card\left(0\right)+\gamma\left\Vert 0\right\Vert _{1}\right)=1.5,
\]
and thus violates definition of a convex function. As a result, problem~\eqref{syncNonconvexForm}
is a nonconvex problem. Note that if we remove the cardinality part
from the objective, then it reduces to a convex optimization problem
because the constraints are affine and the objective is the $\ell_{1}$
norm of an affine transformation of the decision variables \citep[pages 72, 79, 136-137]{Boyd2009}.
Such problems are often called convex-cardinality problem and are
of $\mathcal{NP}$-hard computational complexity in general \citep{Calafiore14}.
An effective yet tractable numerical scheme to achieve a low-cardinality
solution in a convex-cardinality problem is the $\ell_{1}$ norm heuristic,
where $\card(y)$ is replaced by $\|y\|_{1}$, thus converting problem~\eqref{syncNonconvexForm}
into a convex optimization problem. This is described by problem~\eqref{syncConvexForm}
below. The $\ell_{1}$ norm heuristic is supported by extensive numerical
evidence with successful applications to many fields, \emph{e.g.,},
robust estimation in statistics, support vector machine in machine
learning, total variation reconstruction in signal processing, compressed
sensing \emph{etc}. In the next section we show that in our problem
too, the $\ell_{1}$ norm heuristic produces excellent results. Intuitively,
the $\ell_{1}$ norm heuristic works well, because it encourages sparsity
in its arguments by incentivizing exact alignment between regenerative
alignment points with the associated consumptive ones \citep[pages 300-301]{Boyd2009}.
We provide a theoretical justification for the use of $\ell_{1}$
norm in our case as follows.
\begin{prop}
The convex optimization problem described by 
\begin{equation}
\begin{array}{ll}
\textnormal{minimize}\quad\|y\|_{1}\\
\textnormal{subject to}\\
\textnormal{Equations}\eqref{eq:DwellTimeConstraint}-\eqref{eq:TotalTravelTimeConstraints},\eqref{eq:TripTimeConstraintSynModel},\eqref{eq:TurnAroundConstraintSyncModel},\eqref{eq:defOfy},\\
0\leq a_{i}^{t}\leq m,0\leq d_{i}^{t}\leq m,\quad\forall i\in\mathcal{N}^{t},\quad\forall t\in\mathcal{T},
\end{array}\label{syncConvexForm}
\end{equation}
is the best convex approximation of the nonconvex problem~\eqref{syncNonconvexForm}
from below. 
\end{prop}
\begin{proof}
Both problems \eqref{syncConvexForm} and \eqref{syncNonconvexForm}
have the same constraint set, so we need to focus on the objective
only. The best convex approximation of a nonconvex function $f:C\rightarrow\reals$
(where $C$ is any set) from below is given by its convex envelope
$\env f$ on $C$. The function $\env f$ is the largest convex function
that is an under estimator of $f$ on $C$, \emph{i.e.,} 
\[
\env f=\textrm{sup}\{\tilde{f}:C\rightarrow\reals\mid\tilde{f}\text{ is convex and }\tilde{f}\leq f\},
\]
where $\textrm{sup}$ stands for the supremum, , the least upper bound
of the set. The definition implies, $\epi\env f=\conv\epi f$.

From Equation~\eqref{eq:defOfy} we see that $y$ is an affine transformation
of $a$ and $d$, and from the last constraints of problem~\eqref{syncNonconvexForm}
we see that both $a$ and $d$ are upper bounded by $m$, \emph{i.e.},
$\|a\|_{\infty}\leq m$ and $\|d\|_{\infty}\leq m$. So there exists
a positive number $P$ such that $\|y\|_{\infty}\leq P$. As the domain
of $y$ is bounded in an $\ell_{\infty}$ ball with radius $P$, $\env\card(y)=\frac{1}{P}\|y\|_{1}$
\citep[page 321]{Calafiore14}. So, the best convex approximation
of the objective from below is $\frac{1}{P}\|y\|_{1}+\gamma\|y\|_{1}=(\frac{1}{P}+\gamma)\|y\|_{1}$.
As the coefficient $(\frac{1}{P}+\gamma)$ is a constant for a particular
optimization problem, it can be omitted, and thus we arrive at the
claim.
\end{proof}
Using the epigraph approach \citep[pages 143-144]{Boyd2009}, we can
transform the convex problem \eqref{syncConvexForm} into a linear
program as follows. For each $(i,j,t,\overset{\rightharpoonup}{t})\in\overset{\rightharpoonup}{\mathcal{E}}$
and each $(i,j,t,\overset{\leftharpoonup}{t})\in\overset{\leftharpoonup}{\mathcal{E}}$,
we introduce new decision variables $\theta_{ij}^{t\overset{\rightharpoonup}{t}}$
and $\theta_{ij}^{t\overset{\leftharpoonup}{t}}$ respectively, such
that $\theta_{ij}^{t\overset{\rightharpoonup}{t}}\geq|d_{i}^{t}+\vartriangle_{i}^{t}-a_{j}^{\overset{\rightharpoonup}{t}}+\triangledown_{j}^{\overset{\rightharpoonup}{t}}|$
and $\theta_{ij}^{t\overset{\leftharpoonup}{t}}\geq|d_{j}^{\overset{\leftharpoonup}{t}}+\vartriangle_{j}^{\overset{\leftharpoonup}{t}}-a_{i}^{t}+\triangledown_{i}^{t}|$.
Then, the convex optimization problem can be converted into the following
linear problem: 
\begin{equation}
\begin{array}{ll}
\mbox{minimize}\quad\sum_{(i,j,t,\overset{\rightharpoonup}{t})\in\overset{\rightharpoonup}{\mathcal{E}}}{\theta_{ij}^{t\overset{\rightharpoonup}{t}}}+\sum_{(i,j,t,\overset{\leftharpoonup}{t})\in\overset{\leftharpoonup}{\mathcal{E}}}{\theta_{ij}^{t\overset{\leftharpoonup}{t}}}\\
\mbox{subject to}\\
\theta_{ij}^{t\overset{\rightharpoonup}{t}}\geq d_{i}^{t}+\vartriangle_{i}^{t}-a_{j}^{\overset{\rightharpoonup}{t}}+\triangledown_{j}^{\overset{\rightharpoonup}{t}},\qquad\forall(i,j,t,\overset{\rightharpoonup}{t})\in\overset{\rightharpoonup}{\mathcal{E}}\quad\\
\theta_{ij}^{t\overset{\rightharpoonup}{t}}\geq-d_{i}^{t}-\vartriangle_{i}^{t}+a_{j}^{\overset{\rightharpoonup}{t}}-\triangledown_{j}^{\overset{\rightharpoonup}{t}},\qquad\forall(i,j,t,\overset{\rightharpoonup}{t})\in\overset{\rightharpoonup}{\mathcal{E}}\quad\\
\theta_{ij}^{t\overset{\leftharpoonup}{t}}\geq d_{j}^{\overset{\leftharpoonup}{t}}+\vartriangle_{j}^{\overset{\leftharpoonup}{t}}-a_{i}^{t}+\triangledown_{i}^{t},\qquad\forall(i,j,t,\overset{\leftharpoonup}{t})\in\overset{\leftharpoonup}{\mathcal{E}}\\
\theta_{ij}^{t\overset{\leftharpoonup}{t}}\geq-d_{j}^{\overset{\leftharpoonup}{t}}-\vartriangle_{j}^{\overset{\leftharpoonup}{t}}+a_{i}^{t}-\triangledown_{i}^{t},\qquad(i,j,t,\overset{\leftharpoonup}{t})\in\overset{\leftharpoonup}{\mathcal{E}}\\
\text{Equations }\eqref{eq:DwellTimeConstraint}-\eqref{eq:TotalTravelTimeConstraints},\eqref{eq:TripTimeConstraintSynModel},\eqref{eq:TurnAroundConstraintSyncModel},\\
0\leq a_{i}^{t}\leq m,0\leq d_{i}^{t}\leq m\qquad\forall t\in\mathcal{T},\quad\forall i\in\mathcal{N}^{t},
\end{array}\label{finalLinearForm}
\end{equation}
where the decision variables are $a_{i}^{t},d_{i}^{t},\theta_{ij}^{t\overset{\rightharpoonup}{t}}$
and $\theta_{ij}^{t\overset{\leftharpoonup}{t}}$.

\section{Limitations of the optimization model}

\label{sec:lim} In this section, we discuss the limitations of our
model as follows. 
\begin{itemize}
\item We have assumed that the amount by which the trip time is allowed
to vary is on the order of seconds (Assumption \ref{assum:triptime}).
Though this is true for most of the subway systems, there are exceptions
where this assumption may not hold. For example, when a trip between
two cities is considered (especially involving different countries),
the trip is on the order of hours with the acceptable trip time bound
often being on the order of 5-10 minutes and even more in some cases.
In such a scenario, an affine approximation of the energy with respect
to trip time would not be very efficient any more, and our model would
not be suitable for such a case.
\item In Section \ref{Final_Opt} we have have applied two different heuristics
to arrive at a convex optimization problem. At first we have used
$\frac{1}{e}$ heuristic to come up with a tractable description of
the overlapped area between power graphs, and then we have used the
$\ell_{1}$ norm heuristic to approximate a nonconvex objective function
with its convex envelope. So, it is quite likely that the timetable
obtained by solving the convex optimization problem (Problem \ref{finalLinearForm})
would have a worse objective value compared to the original intractable
optimization problem. For this reason our model is energy-efficient,
but not necessarily energy-optimal.
\item The model does not directly address the case of significant delay.
However, we have considered two indirect ways of dealing with it in
practice. 

\begin{itemize}
\item In any automatic train supervision system, which has the responsibility
of implementing the timetable, dwell and velocity regulation are performed
to maintain trains on their proper time. If there is a deviation from
the optimal timetable because of some delay, the ATS performs regulation
to move delayed trains back to the planned optimal timetable timetable.
Thus the system will typically return to a normal state in less than
half an hour after a delay of one minute.
\item Another way is incorporating the change in the system (due to the
delay) as an input data and solving a new but shorter optimization
problem with a time horizon of 1-2 hours which can be done in real
time using our model. While the shorter model is being implemented
we solve the larger optimization problem spanning the rest of the
day. 
\end{itemize}
\end{itemize}
\begin{figure*}[h]
\includegraphics{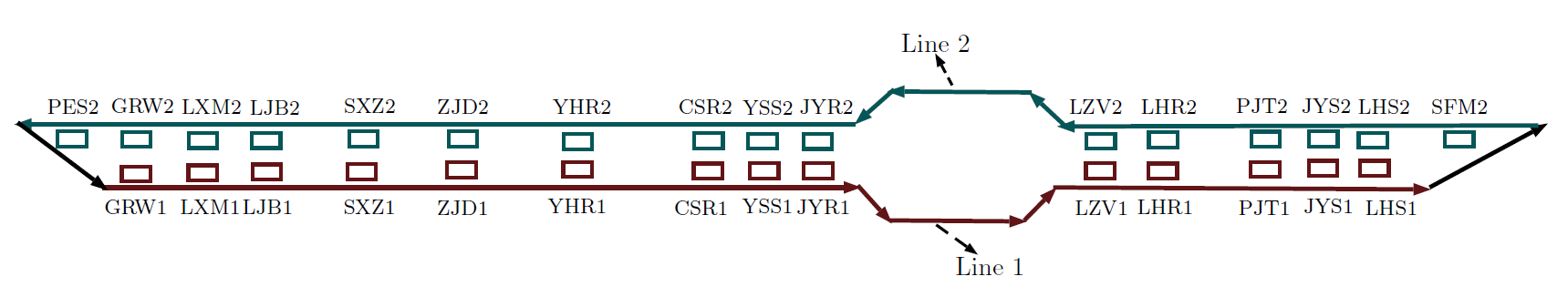} \caption{Railway network considered for numerical study }
\label{shanghaiLine} 
\end{figure*}

\section{Numerical study}

\label{numStud} In this section we apply our model to different problem
instances spanning full service period of one day to service PES2-SFM2
of line 8 of Shanghai Metro network. Shanghai Metro is the world's
largest rapid transit system by route length, second largest by number
of stations(after Beijing), and third oldest rapid transit system
in mainland China. Line 8, opened on December of 2007, is one of the
14 lines of the Shanghai Metro Network. It passes by some of Shanghai's
densest neighborhoods, and has a daily ridership of approximately
1.08 million (2014 peak). This line is 37.4 km long with 28 stations
in operation \citep{wiki}. The service PES2-SFM2 of line 8 of the
Shanghai Metro network is shown in Figure~\ref{shanghaiLine}. There
are two lines in this network: Line 1 and Line 2. There are fourteen
stations in the network denoted by all capitalized words in the figure.
Each station has two platforms each on different train lines, \emph{e.g.,},
LXM is station that has two opposite platforms: LXM1 and LXM2 on Line
1 and Line 2 respectively. The platforms are denoted by rectangles.
The platforms indicated by PES2 and SFM2 are the turn-around points
on Line 2, with the crossing-overs being PES2-GRW1 and LHS1-SFM2.
The number of trains, headway times, speed of the involved trains,
the grades of the tracks and nature of the energy profile of the associated
acceleration and braking phases are different for different instances.

Now we provide some relevant information regarding the railway network
in consideration. The line is 37 km long. The average distance between
two stations in this network is 1.4 km, with the minimum distance
being 738 m (between YHR and ZJD) and maximum distance being 2.6 km
(between PJT and LHR). The slope of the track is in $[-2.00453^{\circ},2.00453^{\circ}]$.
The maximum allowable acceleration of a train is 1.04 $\textrm{m/s}^{2}$
at accelerating phase. The maximum allowable deceleration rate of
a train during coasting phase is -0.2 $\mbox{\textrm{m/s}}^{2}$.
The maximum allowable deceleration of a train during braking phase
is -0.8 $\textrm{m/s}^{2}$. The conversion factor from electricity
to kinetic energy is 0.9, and the conversion factor from kinetic to
regenerative braking energy is 0.76. The transmission loss factor
of regenerative electricity is 0.1. The mass of the train is in $[229370,361520]$
kg, with the average mass being $295445$ kg.

The data on speed limit is described by Table~\ref{table-speed-limit}
in Appendix. The speed limit data are based on grade and curvature
of the tracks, and operational constraints present in the system.
For the railway network, the tracks have piece-wise speed limits for
trains, \emph{i.e.,} each track (except CSR2-YHR2) is divided into
multiple segments, where each segment has a constant speed limit.
The track CSR2-YHR2 has only one segment (itself) with a speed limit
of 60 km/h. Table~\ref{table-speed-limit} is provided to us by Thales
Canada Inc. 

The numerical study was executed on a Intel Core i7-46400 CPU with
8 GB RAM running Windows 8.1 Pro operating system. For modelling the
problem, we have used \texttt{JuMP} - an open source algebraic modelling
language embedded in programming language \texttt{Julia} \citep{Lubin2013}.
Within our \texttt{JuMP} code we have called academic version of Gurobi
Optimizer 6.0 as the solver. We have implemented an interior point
algorithm because of the underlying sparsity in the data structure.
As mentioned before, a measure of the quality of affine fittings using
least-squares approach is given by the \emph{coefficient of determination},
which can vary between 0 to 1, with 0 being the worst and 1 being
the best \citep[page 518]{Mendenhall2012}. In our numerical study,
the average coefficient of determination of the affine fittings for
energy versus trip times over all different trips and all trains is
found to be 0.9483 with a standard deviation of 0.05, which justifies
our approach.

The duration of the timetables is eighteen hours which is the full
service period of the railway network. We have considered eleven different
instances with varying average headway times and number of trains.
The number of trains increases as the average headway time decreases,
where the relation between them can be determined from Equation \ref{eq:headway_calculation}.
The results of the numerical study are shown in Table~\ref{Result_Table}.

We can see that, in all of the cases our model has found the optimal
timetables very quickly, the largest runtime being 12.58s. To the
best of our knowledge, this model is the only one to calculate energy-efficient
railway timetable spanning an entire day, the next largest being 6
hours only \citep{DasGuptaACC15} with a much larger computation time
for smaller sized problems. After we get the final timetable, we calculate
the total \emph{effective energy consumption} by all trains involved
in SPSTPs and compare it with the original timetables. The effective
energy consumption of a train during a trip is defined as the difference
between the total energy required to make a trip and the amount of
energy that is being supplied by a braking train during synchronization
process. So, the effective energy consumption is the energy that will
be consumed from the electrical substations.

\textcolor{black}{Now we briefly point out how the total energy consumption
is calculated theoretically. Consider the trip of any train $t$ of
mass $m_{t}$ from one platform to the next one. Assume the trip time
to make the trip is $T$. At time instance $\tau$, the acceleration,
speed, position, and resistive acceleration is denoted by $u_{t}(\tau),v_{t}(\tau)$,
$s_{t}(\tau)$ and $-r\left(v_{t}(\tau)\right)$, respectively. The
resistive acceleration is given by the }\textcolor{black}{\emph{Davis
formula }}\textcolor{black}{\citep{Davis1926} 
\[
r\left(v_{t}(\tau)\right)=a_{0}+a_{1}v_{t}(\tau)+a_{2}v_{t}^{2}(\tau),
\]
 where $a_{0},a_{1},a_{2}$ are known positive numbers. At time instance
$\tau$, the net acceleration of the train is equal to $u_{t}(\tau)-r\left(v_{t}(\tau)\right)$,
where $u_{t}(\tau)$ bounded by magnitude constraints which are non-increasing
and depends on the speed of the train. Only the positive acceleration
force consumes energy and negative acceleration during braking produces
regenerative energy. Define, $u_{t}^{+}(\tau)=\max\{0,u_{t}(\tau)\}$,
and $u_{t}^{-}(\tau)=\max\{0,-u(\tau)\}$ which represents the positive
and negative part of $u_{t}(\tau)$, respectively. The negative acceleration
associated with regenerative braking is denoted by $u_{t,\text{regen}}^{-}(\tau)$.
So, the total energy consumed by the train while making this trip
is given by the integral: $\int_{0}^{T}m_{t}u_{t}^{+}(\tau)s_{t}(\tau)d\tau$
and the total regenerative energy produced by the train is given by
$\int_{0}^{T}m_{t}u_{t,\text{regen}}^{-}(\tau)s_{t}(\tau)d\tau.$}

\textcolor{black}{Consider a suitable train pair, say trains $t$
and $\tilde{t}$ associated with the SPSTP $(i,j,t,\tilde{t})$, with
train $t$ accelerating and train $\tilde{t}$ braking. After we have
calculated the timetable, it will provide us with the duration time
during which they are synchronized. Let us denote the times for the
beginning and end of the synchronization process by $T_{1}$ and $T_{2}$.
For any $x\in\mathbf{R}$, denote $x^{+}=\max\{x,0\}$. Then the effective
energy consumption of associated with this train pair is denoted by,
\[
\int_{T_{1}}^{T_{2}}\left(m_{t}u_{t}^{+}(\tau)s_{t}(\tau)-m_{\tilde{t}}u_{\tilde{t},\text{regen}}^{-}(\tau)s_{\tilde{t}}(\tau)\right)^{+}d\tau,
\]
The physical interpretation for the above integral is as follows.
If for some reason we have more regenerative energy can be provided
than needed by accelerating train (though the likelihood of the occurrence
of this case is very low in practice), then the extra energy is burned
via resistive braking. Similarly we can define the effective energy
consumption with train $\tilde{t}$ accelerating and train t braking
as follows: 
\[
\int_{T_{1}}^{T_{2}}\left(m_{\tilde{t}}u_{\tilde{t}}^{+}(\tau)s_{\tilde{t}}(\tau)-m_{t}u_{t,\text{regen}}^{-}(\tau)s_{t}(\tau)\right)^{+}d\tau.
\]
 So, the effective energy consumption effective energy consumption
associated with the SPSTP $(i,j,t,\tilde{t})$ is denoted by: 
\[
E_{\text{effective}}^{(i,j,t,\tilde{t})}=\begin{cases}
\int_{T_{1}}^{T_{2}}\left(m_{t}u_{t}^{+}(\tau)s_{t}(\tau)-m_{\tilde{t}}u_{\tilde{t},\text{regen}}^{-}(\tau)s_{\tilde{t}}(\tau)\right)^{+}d\tau, & \;\text{if train \ensuremath{t} is accelerating and train \ensuremath{\tilde{t}} is braking}\\
\int_{T_{1}}^{T_{2}}\left(m_{\tilde{t}}u_{\tilde{t}}^{+}(\tau)s_{\tilde{t}}(\tau)-m_{t}u_{t,\text{regen}}^{-}(\tau)s_{t}(\tau)\right)^{+}d\tau, & \;\text{if train \ensuremath{t} is braking and train \ensuremath{\tilde{t}} is accelerating}
\end{cases}
\]
 In practice, these integrations are performed numerically for which
robust and fast packages exist }\textcolor{black}{\emph{e.g.,}}\textcolor{black}{{}
\citep{cubature} in our case. The total effective energy consumption
over all SPSTPs is given by: 
\[
E_{\text{effective}}=\sum_{(i,j,t,\tilde{t})\in\mathcal{E}}E_{\text{effective}}^{(i,j,t,\tilde{t})}.
\]
}

The original timetables, which we compare the final timetables with,
are provided by Thales Canada Inc. It should be noted that, the number
of trains $\mathcal{T}$ is fixed for each of the instances. The energy
calculation is done using \texttt{SPSIM}, which is a proprietary software
owned by Thales Canada Inc \citep{selTrac}, and \texttt{Cubature},
which is an open-source \texttt{Julia} package written by Steven G.
Johnson that uses an adaptive algorithm for the approximate calculation
of multiple integrals \citep{cubature}. \texttt{SPSIM} calculates
the power versus time graphs of all the active trains for the original
and optimal timetables. \texttt{Cubature} is used to calculate the
effective area under the power versus time graphs to determine 1)
the total energy required by the trains during the trips, 2) the total
transferred regenerative energy during the SPSTPs, and 3) the effective
energy consumption as the difference of the first two quantities.
The effective energy consumption of the optimal timetables in comparison
with the original ones is reduced quite significantly - even in the
worst case, the reduction in effective energy consumption is 19.27\%,
with the best case corresponding to 21.61\%.

\begin{table*}[tp]
\centering{} {\scriptsize{}\caption{Results of the numerical study performed to line 8 of Shanghai Metro
network}
\label{Result_Table} }%
\begin{tabular}{|p{1cm}|p{1cm}|p{1cm}|p{1cm}|p{1cm}|p{1cm}|p{1cm}|p{2cm}|p{2cm}|p{2cm}|}
\hline 
{\scriptsize{}Number of trains } & {\scriptsize{}Number of constraints step 1 } & {\scriptsize{}Number of variables step 1 } & {\scriptsize{}step 1 CPU time (s) } & {\scriptsize{}Number of constraints step 2 } & {\scriptsize{}Number of Variable step 2 } & {\scriptsize{}step 2 CPU Time (s) } & {\scriptsize{}Initial effective energy consumption (kWh) } & {\scriptsize{}Final effective energy consumption (kWh) } & {\scriptsize{}Reduction in effective energy consumption }\tabularnewline
\hline 
\hline 
{\scriptsize{}1000 } & {\scriptsize{}91998 } & {\scriptsize{}30060 } & {\scriptsize{}3.24 } & {\scriptsize{}116558 } & {\scriptsize{}34871 } & {\scriptsize{}6.03 } & {\scriptsize{}250951.3 } & {\scriptsize{}201658.7 } & {\scriptsize{}19.64 \% }\tabularnewline
\hline 
{\scriptsize{}1032 } & {\scriptsize{}94944 } & {\scriptsize{}31022 } & {\scriptsize{}3.03 } & {\scriptsize{}120394 } & {\scriptsize{}36038 } & {\scriptsize{}5.45 } & {\scriptsize{}261994.5 } & {\scriptsize{}208558.1 } & {\scriptsize{}20.40 \% }\tabularnewline
\hline 
{\scriptsize{}1066 } & {\scriptsize{}98074 } & {\scriptsize{}32044 } & {\scriptsize{}3.96 } & {\scriptsize{}124494 } & {\scriptsize{}37290 } & {\scriptsize{}5.62 } & {\scriptsize{}272486.7 } & {\scriptsize{}215896.7 } & {\scriptsize{}20.77 \% }\tabularnewline
\hline 
{\scriptsize{}1100 } & {\scriptsize{}101204 } & {\scriptsize{}33066 } & {\scriptsize{}3.47 } & {\scriptsize{}129284 } & {\scriptsize{}38887} & {\scriptsize{}5.39 } & {\scriptsize{}288677.5 } & {\scriptsize{}229091.8 } & {\scriptsize{}20.64 \% }\tabularnewline
\hline 
{\scriptsize{}1132 } & {\scriptsize{}104150 } & {\scriptsize{}34028 } & {\scriptsize{}3.15 } & {\scriptsize{}133354 } & {\scriptsize{}40171 } & {\scriptsize{}6.69 } & {\scriptsize{}308924.4 } & {\scriptsize{}243672 } & {\scriptsize{}21.12 \% }\tabularnewline
\hline 
{\scriptsize{}1166 } & {\scriptsize{}107280 } & {\scriptsize{}35050 } & {\scriptsize{}2.84 } & {\scriptsize{}137322 } & {\scriptsize{}41357 } & {\scriptsize{}6.67 } & {\scriptsize{}322612.7 } & {\scriptsize{}256288.7 } & {\scriptsize{}20.56 \% }\tabularnewline
\hline 
{\scriptsize{}1198 } & {\scriptsize{}110226 } & {\scriptsize{}36012 } & {\scriptsize{}2.96 } & {\scriptsize{}141322 } & {\scriptsize{}42606 } & {\scriptsize{}7.16 } & {\scriptsize{}329388.2 } & {\scriptsize{}262205.6 } & {\scriptsize{}20.40\% }\tabularnewline
\hline 
{\scriptsize{}1232 } & {\scriptsize{}113356 } & {\scriptsize{}37034 } & {\scriptsize{}4.04 } & {\scriptsize{}145756 } & {\scriptsize{}44025 } & {\scriptsize{}7.61 } & {\scriptsize{}354050.2 } & {\scriptsize{}277536.7 } & {\scriptsize{}21.61\% }\tabularnewline
\hline 
{\scriptsize{}1266 } & {\scriptsize{}116486 } & {\scriptsize{}38056 } & {\scriptsize{}3.84 } & {\scriptsize{}149868 } & {\scriptsize{}45283 } & {\scriptsize{}8.74 } & {\scriptsize{}368901.4 } & {\scriptsize{}297815 } & {\scriptsize{}19.27 \% }\tabularnewline
\hline 
{\scriptsize{}1298 } & {\scriptsize{}119432 } & {\scriptsize{}39018 } & {\scriptsize{}3.93 } & {\scriptsize{}153480 } & {\scriptsize{}46338 } & {\scriptsize{}7.62 } & {\scriptsize{}366488.4 } & {\scriptsize{}293068.8 } & {\scriptsize{}20.03 \% }\tabularnewline
\hline 
{\scriptsize{}1332 } & {\scriptsize{}122562 } & {\scriptsize{}40040 } & {\scriptsize{}4.22 } & {\scriptsize{}157752 } & {\scriptsize{}47676 } & {\scriptsize{}8.02 } & {\scriptsize{}379700.8 } & {\scriptsize{}300910.1 } & {\scriptsize{}20.75 \% }\tabularnewline
\hline 
\end{tabular}
\end{table*}

\section{Conclusion}

\label{Conclusion} In this paper we have proposed a novel two-step
linear optimization model to calculate an energy-efficient timetable
in modern metro railway networks. The objective is to minimize the
total electrical energy consumption of all trains and to maximize
the utilization of regenerative energy produced by braking trains.
In contrast to other existing models, this model is computationally
the most tractable one. We have applied our optimization model to
eleven different instances of service PES2-SFM2 of line 8 of Shanghai
Metro network. All instances span the full service period of one day
(18 hours) with thousands of active trains. For all instances our
model has found optimal timetables in less than 13s with significant
reductions in the effective energy consumption. Code based on our
optimization model has been integrated with the industrial timetable
compiler of Thales Inc.

\section*{Acknowledgments}

This work was supported by NSERC-CRD and Thales, Inc (CRDPJ 461180
-13). The authors acknowledge helpful discussions with Professor J.
Christopher Beck, Department of Mechanical \& Industrial Engineering,
University of Toronto.

\section*{Appendix }

\begin{longtable}{|c|c|c|c|}
\caption{Speed limit for{\footnotesize{} }line 8 of Shanghai Metro network{\footnotesize{}\label{table-speed-limit}}}
\tabularnewline
\hline 
\textbf{Origin-Destination}  & \textbf{Start (m)}  & \textbf{End (m)}  & \textbf{Speed limit (km/h)} \tabularnewline
\endfirsthead
\hline 
CSR1-YSS1  & 0.0  & 143.5  & 60 \tabularnewline
\hline 
 & 143.5  & 1004.6  & 70 \tabularnewline
\hline 
 & 1004.6  & 1138.2  & 60 \tabularnewline
\hline 
\hline 
CSR2-YHR2  & 0.0  & 910.0  & 60 \tabularnewline
\hline 
\hline 
GRW1-LXM1  & 0.0  & 153.3  & 60 \tabularnewline
\hline 
 & 153.3  & 870.1  & 70 \tabularnewline
\hline 
 & 870.1  & 1006.9  & 60 \tabularnewline
\hline 
\hline 
GRW2-PES2  & 0.0  & 173.1  & 60 \tabularnewline
\hline 
 & 173.1  & 636.4  & 70 \tabularnewline
\hline 
 & 636.4  & 769.5  & 60 \tabularnewline
\hline 
\hline 
JYR1-LZV1  & 0.0  & 1366.7  & 60 \tabularnewline
\hline 
 & 1366.7  & 2220.6  & 65 \tabularnewline
\hline 
 & 2220.6  & 2357.3  & 60 \tabularnewline
\hline 
\hline 
JYR2-YSS2  & 0.0  & 143.4  & 60 \tabularnewline
\hline 
 & 143.4  & 1388.9  & 70 \tabularnewline
\hline 
 & 1388.9  & 1522.3  & 60 \tabularnewline
\hline 
\hline 
JYS1-LHS1  & 0.0  & 140.0  & 60 \tabularnewline
\hline 
 & 140.0  & 829.2  & 75 \tabularnewline
\hline 
 & 829.2  & 1202.3  & 60\tabularnewline
\hline 
\hline 
JYS2-PJT2  & 0.0  & 140.0  & 60\tabularnewline
\hline 
 & 140.0  & 371.1  & 70 \tabularnewline
\hline 
 & 371.1  & 1081.9  & 75 \tabularnewline
\hline 
 & 1081.9  & 1249.8  & 70 \tabularnewline
\hline 
 & 1249.8  & 1386.2  & 60 \tabularnewline
\hline 
\hline 
LHR1-PJT1  & 0.0  & 140.1  & 60 \tabularnewline
\hline 
 & 140.1  & 766.2  & 70 \tabularnewline
\hline 
 & 766.2  & 1623.4  & 75 \tabularnewline
\hline 
 & 1623.4  & 1805.9  & 70 \tabularnewline
\hline 
 & 1805.9  & 2374.4  & 75 \tabularnewline
\hline 
 & 2374.4  & 2487.8  & 70 \tabularnewline
\hline 
 & 2487.8  & 2624.3  & 60 \tabularnewline
\hline 
\hline 
LHR2-LZV2  & 0.0  & 139.8  & 60 \tabularnewline
\hline 
 & 139.8  & 2457.3  & 70 \tabularnewline
\hline 
 & 2457.3  & 2594.1  & 60 \tabularnewline
\hline 
\hline 
LHS1-SFM1  & 0.0  & 140.0  & 60 \tabularnewline
\hline 
 & 140.0  & 1220.1  & 70 \tabularnewline
\hline 
 & 1220.1  & 1357.4  & 60 \tabularnewline
\hline 
\hline 
LHS2-JYS2  & 0.0  & 186.7  & 60 \tabularnewline
\hline 
 & 186.7  & 853.8  & 75 \tabularnewline
\hline 
 & 853.8  & 1064.4  & 70\tabularnewline
\hline 
 & 1064.4  & 1200.8  & 60 \tabularnewline
\hline 
\hline 
LJB1-SXZ1  & 0.0  & 140.1  & 60\tabularnewline
\hline 
 & 140.1  & 1027.4  & 70 \tabularnewline
\hline 
 & 1027.4  & 1167.2  & 60 \tabularnewline
\hline 
\hline 
LJB2-LXM2  & 0.0  & 140.1  & 60\tabularnewline
\hline 
 & 140.1  & 693.3  & 70 \tabularnewline
\hline 
 & 693.3  & 830.2  & 60 \tabularnewline
\hline 
\hline 
LXM1-LJB1  & 0.0  & 140.0  & 60 \tabularnewline
\hline 
 & 140.0  & 689.3  & 70 \tabularnewline
\hline 
 & 689.3  & 826.1  & 60 \tabularnewline
\hline 
\hline 
LXM2-GRW2  & 0.0  & 140.0  & 60 \tabularnewline
\hline 
 & 140.0  & 855.5  & 70 \tabularnewline
\hline 
 & 855.5  & 1005.2  & 60 \tabularnewline
\hline 
\hline 
LZV1-LHR1  & 0.0  & 140.1  & 60 \tabularnewline
\hline 
 & 140.1  & 1901.1  & 70 \tabularnewline
\hline 
 & 1901.1  & 2199.1  & 75 \tabularnewline
\hline 
 & 2199.1  & 2456.3  & 70 \tabularnewline
\hline 
 & 2456.3  & 2592.8  & 60 \tabularnewline
\hline 
\hline 
LZV2-JYR2  & 0.0  & 143.3  & 60 \tabularnewline
\hline 
 & 143.3  & 1007.8  & 65 \tabularnewline
\hline 
 & 1007.8  & 2338.1  & 60 \tabularnewline
\hline 
\hline 
PES1-GRW1  & 0.0  & 140.0  & 60 \tabularnewline
\hline 
 & 140.0  & 561.7  & 70 \tabularnewline
\hline 
 & 561.7  & 761.5  & 60\tabularnewline
\hline 
\hline 
PJT1-JYS1  & 0.0  & 143.3  & 60 \tabularnewline
\hline 
 & 143.3  & 374.6  & 70 \tabularnewline
\hline 
 & 374.6  & 1089.8  & 75 \tabularnewline
\hline 
 & 1089.8  & 1250.2  & 70 \tabularnewline
\hline 
 & 1250.2  & 1386.7  & 60 \tabularnewline
\hline 
\hline 
PJT2-LHR2  & 0.0  & 143.4  & 60 \tabularnewline
\hline 
 & 143.4  & 355.8  & 70 \tabularnewline
\hline 
 & 355.8  & 829.3  & 75 \tabularnewline
\hline 
 & 829.3  & 1039.2  & 70 \tabularnewline
\hline 
 & 1039.2  & 1858.3  & 75 \tabularnewline
\hline 
 & 1858.3  & 2488.9  & 70 \tabularnewline
\hline 
 & 2488.9  & 2622.1  & 60 \tabularnewline
\hline 
\hline 
SFM2-LHS2  & 0.0  & 140.3  & 60 \tabularnewline
\hline 
 & 140.3  & 373.0  & 70 \tabularnewline
\hline 
 & 373.0  & 742.3  & 75 \tabularnewline
\hline 
 & 742.3  & 1225.1  & 70 \tabularnewline
\hline 
 & 1225.1  & 1358.3  & 60 \tabularnewline
\hline 
\hline 
SXZ1-ZJD1  & 0.0  & 140.1  & 60 \tabularnewline
\hline 
 & 140.1  & 647.2  & 65 \tabularnewline
\hline 
 & 647.2  & 1699.3  & 70 \tabularnewline
\hline 
 & 1699.3  & 2039.3  & 60 \tabularnewline
\hline 
\hline 
SXZ2-LJB2  & 0.0  & 160.1  & 60 \tabularnewline
\hline 
 & 160.1  & 1027.5  & 70 \tabularnewline
\hline 
 & 1027.5  & 1164.2  & 60 \tabularnewline
\hline 
\hline 
YHR1-CSR1  & 0.0  & 143.4  & 60\tabularnewline
\hline 
 & 143.4  & 773.7  & 70 \tabularnewline
\hline 
 & 773.7  & 910.3  & 60 \tabularnewline
\hline 
\hline 
YHR2-ZJD2  & 0.0  & 140.0  & 60 \tabularnewline
\hline 
 & 140.0  & 601.3  & 70 \tabularnewline
\hline 
 & 601.3  & 738.0  & 60 \tabularnewline
\hline 
\hline 
YSS1-JYR1  & 0.0  & 140.2  & 60 \tabularnewline
\hline 
 & 140.2  & 664.7  & 70 \tabularnewline
\hline 
 & 664.7  & 987.7  & 75 \tabularnewline
\hline 
 & 987.7  & 1389.2  & 70 \tabularnewline
\hline 
 & 1389.2  & 1525.9  & 60 \tabularnewline
\hline 
\hline 
YSS2-CSR2  & 0.0  & 430.4  & 60 \tabularnewline
\hline 
 & 430.4  & 1014.8  & 70 \tabularnewline
\hline 
 & 1014.8  & 1151.6  & 54 \tabularnewline
\hline 
\hline 
ZJD1-YHR1  & 0.0  & 140.1  & 60 \tabularnewline
\hline 
 & 140.1  & 605.9  & 70 \tabularnewline
\hline 
 & 605.9  & 742.5  & 60 \tabularnewline
\hline 
\hline 
ZJD2-SXZ2  & 0.0  & 353.8  & 60 \tabularnewline
\hline 
 & 353.8  & 1393.7  & 70 \tabularnewline
\hline 
 & 1393.7  & 1910.0  & 65 \tabularnewline
\hline 
 & 1910.0  & 2043.3  & 60 \tabularnewline
\hline 
\end{longtable}

\bibliographystyle{IEEEtran}

\end{document}